\documentclass{article}
\usepackage{amsmath, amsfonts, amssymb, verbatim, hyperref}
\usepackage[all,cmtip]{xy}
\usepackage{multirow} 
\usepackage{lscape}
\usepackage{longtable}

\newtheorem{theorem}{Theorem}[section]
\newtheorem{lemma}[theorem]{Lemma}
\newtheorem{Definition}[theorem]{Definition}
\newtheorem{proposition}[theorem]{Proposition}
\newenvironment{proof}[1][]{\textbf{Proof#1.}}{$\Box$ }

\DeclareMathOperator{\id}{id}
\newcommand{\End}{\operatorname{End}}
\newcommand{\onlineVersionOnly}[1]{#1}

\title{Computing parabolically induced embeddings of semisimple complex Lie algebras in Weyl algebras}

\author{Todor Milev  }

\begin{document}

\maketitle

\begin{abstract}
An arbitrary proper parabolic subalgebra ${\mathfrak p}$ of a simple complex Lie algebra ${\mathfrak g}$ induces an embedding ${\mathfrak g}\hookrightarrow \mathbb W_n$, and more generally an embedding ${\mathfrak g}\hookrightarrow \mathbb W_n\otimes \End V$, where $\mathbb W_n$ is the Weyl algebra in $n$ variables, $n$ is the dimension of the nilradical of ${\mathfrak p}$, and $V$ is an arbitrary ${\mathfrak p}$-module. We give an elementary proof of this known fact, report on a computer program computing the embeddings, and tabulate exceptional Lie algebra embeddings $G_2 \hookrightarrow \mathbb W_5$,  $F_4 \hookrightarrow \mathbb W_{15}$, $E_6 \hookrightarrow \mathbb W_{16}$, $E_7 \hookrightarrow\mathbb W_{27}$, $E_8 \hookrightarrow \mathbb W_{57}$ arising in this fashion.
\end{abstract}

\section{Introduction}
Let ${\mathfrak g}$ be a semisimple Lie algebra. In this article we present a computational technique for realizing generalized Verma ${\mathfrak g}$-modules via differential operators. For simple ${\mathfrak g}\neq sp(2n)$, the resulting differential operator realizations of ${\mathfrak g}$ provide candidates for differential operator realizations with minimal number of variables (see \cite{Joseph:MinimalRealizations}). At the same time, these differential operator realizations serve as the first computational step in the ``$\mathcal F$-method'' for computing generalized Verma module branching rules. 

Our computational technique is contained in the proof of Proposition \ref{propGenVermaIsInWeylPlusEnd}. The proposition is known, see \cite{Berline:algebrasDiffOps}. Our contribution consists in giving a computational proof, well suited for computer realization, and writing a computer program to tabulate the corresponding exceptional Lie algebra differential operator realizations (Theorem \ref{thExceptionalsinWn} and Table \ref{tableTheTable} below). 
 
Let ${\mathfrak p}', {\mathfrak n}'\subset {\mathfrak g}$ be two Lie subalgebra of ${\mathfrak g}$ such that ${\mathfrak n}$ is nilpotent and ${\mathfrak g}$ is the direct vector space sum of ${\mathfrak n}'$ and ${\mathfrak p}'$. Let $V$ be a representation of ${\mathfrak p}'$. It is then known that there exists a Lie algebra homomorphism mapping ${\mathfrak g}$ to the Lie algebra $\mathbb W_n \otimes \End V$, where $\mathbb W_n$ is the Weyl algebra of differential operators with polynomial coefficients in $n=\dim{\mathfrak n}' $ variables.  This is explained, for example, in \cite{Joseph:MinimalRealizations} in the remarks before Lemma 3.3 there. In the particular case that $V$ is a trivial representation, this map may be identified with a homomorphism from ${\mathfrak g}$ into $\mathbb W_n$.

Suppose now that ${\mathfrak g}$ is a simple complex Lie algebra. Then a homomorphism of the above form must be an embedding. In the case that ${\mathfrak p}'={\mathfrak p}$ is a parabolic subalgebra with reductive Levi part ${\mathfrak l}$ and ${\mathfrak n}'={\mathfrak n}_-$ is the subalgebra opposite to the nilradical of ${\mathfrak p}$, in Proposition \ref{propGenVermaIsInWeylPlusEnd} we construct the corresponding embedding  
\begin{equation}
\label{eqEmbeddingOfSimpleLieAlg}
{\mathfrak g}\hookrightarrow\mathbb W_n\otimes \End V_\lambda({\mathfrak l})\quad .
\end{equation}
We report a computer program\footnote{the program has free and open source code and an online interface} written by the author computes the embeddings \eqref{eqEmbeddingOfSimpleLieAlg}, as described in Section \ref{secComputerWork}. We have provided an online interface to our program, \cite{Milev:vpf}, which allows an user to compute arbitrary examples of \eqref{eqEmbeddingOfSimpleLieAlg} within computational range.  The dimension of $V_\lambda({\mathfrak l})$ is restricted by computational power; $\dim V_\lambda({\mathfrak l})\leq 1000$ is within practical limits on our system (when running on a 32 bit personal computer). The program is written in the programming language C++, within the project ``vector partition function''. Table \ref{tableTheTable} was generated by our C++ program. In Section \ref{secTheConstructionCaseG2} we illustrate the details of the construction for $G_2$, as printed out by our program.

\subsection{Relationship to minimal realizations of simple Lie algebras}
 
If we specialize our computational technique to the case of parabolic subalgebra of maximal dimension and $\lambda=0$, we obtain a case that is key to the study of minimal embeddings of ${\mathfrak g}\neq sp(2n)$ into Weyl algebras. We recall that an embedding of a Lie algebra in the Weyl algebra of $n$ variables $\mathbb W_n$ is minimal if it uses the minimal number of variables $n$. In the case of an exceptional Lie algebra and proper parabolic subalgebra of maximal dimension, we get the embeddings 
\begin{equation}\label{eqNumberOfVariablesNeededForExceptionals}
E_8\hookrightarrow \mathbb W_{57}, E_7 \hookrightarrow \mathbb W_{27}, E_6\hookrightarrow \mathbb W_{16}, F_4\hookrightarrow \mathbb W_{15}, G_2\hookrightarrow \mathbb W_{5}\quad .
\end{equation}

\cite{Joseph:MinimalRealizations} proves that a minimal realization $sl(n+1)\hookrightarrow \mathbb W_n$ is achieved by the construction presented here and is not achieved for $sp(2n)$ (for $sp(2n)$ \cite{Joseph:MinimalRealizations} proves the minimal realization is the Shale-Weil representation). If we change the problem to allow localizations $\overline{\mathbb W}_k$ of $\mathbb W_k$, \cite{Joseph:MinDimNilpotentOrbit}, \cite{Joseph:MinimalRealizations} solve the problem of minimal realizations in $\overline{\mathbb W}_k$. Except for $sl(n+1)$, the realizations in $\overline{\mathbb W}_k$ use a smaller number of variables  than the construction discussed in the present paper. However, if using $\mathbb W_n$ rather than localizations, \cite{Joseph:MinDimNilpotentOrbit}, \cite{Joseph:MinimalRealizations} do not resolve the question of whether $so(n)\hookrightarrow \mathbb W_{n-2}$ and the embeddings \eqref{eqNumberOfVariablesNeededForExceptionals} are minimal realizations; up to our knowledge, this remains an open question.

Embeddings \eqref{eqNumberOfVariablesNeededForExceptionals} are presented in the following theorem.
\begin{theorem}\label{thExceptionalsinWn}
Let ${\mathfrak g}\subset \mathbb W_n\otimes \id $ be the Lie algebra generated by the differential operators indicated in the last column of Table \ref{tableTheTable}, where $n$ is the number indicated in the second column. Then ${\mathfrak g}$ is isomorphic to the exceptional Lie algebra indicated in the first column of the table.
\end{theorem}
While the number of variables in the realizations \eqref{eqNumberOfVariablesNeededForExceptionals} were stated in \cite{Joseph:MinimalRealizations} and \cite{Joseph:MinDimNilpotentOrbit}, the present article is the only one (up to our knowledge) to explicitly present differential operator realizations of the exceptional Lie algebras $F_4, E_6, E_7, E_8$ in the stated number of variables.

We note that the results of Section \ref{secTheoretical} are valid over a field of arbitrary characteristic - the only division operation required by Section \ref{secTheoretical} is the division by a monic polynomial in \eqref{eqPolyDivForWeyl}. So long as the structure constants with respect to a Chevalley-Weyl basis of the Lie algebra do not vanish (which does happen in characteristic 2 and 3  for $G_2$ and in characteristic 2 for $F_4$), all algorithms described in Section \ref{secComputerWork} carry through directly. Furthermore, the Lie algebra generated over $\mathbb Z$ by the differential operators in Theorem \ref{thExceptionalsinWn} is spanned by differential operators with integral coefficients. Clearly those involve only finitely many primes, and therefore Theorem \ref{thExceptionalsinWn} remains valid over all fields of sufficiently large finite characteristic.

\subsection{Relationship to the ``$\mathcal F$-method''}
The technique presented here was originally intended as a tool for solving the generalized Verma module branching problem over a semisimple subalgebra as formulated in \cite{MilevSomberg:BranchingShort}. Our technique provides the starting differential operators used in the ``$\mathcal F$-method'' for computing branchings of generalized Verma modules, see  \cite[\S 2.4]{OerstedKobayashiSombergSoucek:Fmethod}. The software described in the present paper was used in part of the computation of certain families of generalized Verma module branchings of $so(7)$ over $G_2$, \cite{MilevSomberg:so7g2Fmethod}.


\textbf{Acknowledgment}: The author gratefully acknowledges the support by the Czech Grant Agency through the grant GA CR P 201/12/G028. 

\section{Notation and preliminaries}\label{secNotationPreliminaries}
${\mathfrak g}$ denotes a semisimple Lie algebra over the complex numbers. Let ${\mathfrak h}\subset {\mathfrak g}$ be a Cartan subalgebra, and let ${\mathfrak g}={\mathfrak h}\oplus \bigoplus_{\alpha\in \Delta({\mathfrak g})} {\mathfrak g}^{\alpha}$ be the corresponding root space decomposition of ${\mathfrak g}$. Let $g_\alpha\in {\mathfrak g}^{\alpha}$ be a Chevalley-Weyl basis of ${\mathfrak g}$ corresponding to this decomposition. In a Chevalley-Weyl basis, the structure constants of ${\mathfrak g}$ are integral. 

We denote by ${\mathfrak b},{\mathfrak p}$ respectively a fixed Borel and a fixed parabolic subalgebra, such that ${\mathfrak h}\subset{\mathfrak b}\subset{\mathfrak p} $. We denote by ${\mathfrak l}$ the reductive Levi part of ${\mathfrak p}$ that contains ${\mathfrak h}$, by ${\mathfrak n}$ the nilradical of ${\mathfrak p}$ and by ${\mathfrak n}_-$ the subalgebra generated by the root spaces opposite to the root spaces that generate ${\mathfrak n}$. 

For a dominant integral weight $\lambda\in {\mathfrak h}^*$ we denote by $V_\lambda({\mathfrak l})$ the irreducible finite-dimensional representation of ${\mathfrak l}$ with highest weight $\lambda$. We equip $V_\lambda({\mathfrak l})$ with trivial action of ${\mathfrak n}$ and denote by $M_\lambda({\mathfrak g}, {\mathfrak p})$ the corresponding generalized Verma module induced by the action of ${\mathfrak p}$ on $V_\lambda({\mathfrak l})$, i.e., 
\[
M_\lambda({\mathfrak g},{\mathfrak p}):= U({\mathfrak g})\otimes_{U({\mathfrak p})} V_\lambda({\mathfrak l})\quad .
\]

The PBW theorem implies the vector space isomorphism
\begin{equation}\label{eqGenVermaIsUnTimesV}
M_\lambda({\mathfrak g},{\mathfrak p}) \simeq U({\mathfrak n}_-) \otimes V_{\lambda}({\mathfrak l})\quad .
\end{equation}

Let $\mathbb W_n$ denote the Weyl algebra in $n $ variables over $\mathbb Q$ where $\dim{\mathfrak n}=n$. We denote its generators by $\partial_1,\dots, \partial_n $, $x_1, \dots, x_n$, and as usual we define its relations to be those of an associative algebra with identity together with
\[
\begin{array}{rcl}
x_i x_j&=&x_jx_i, \\
\partial_i\partial_j&=&\partial_j\partial_i,\\
\partial_j x_i -x_i\partial_j&=&[\partial_j, x_i]=\delta_{ij}:=\left\{\begin{array}{ll}1 & \mathrm {if~} i=j\\0&\mathrm{otherwise.} \end{array} \right.
\end{array}
\] 
Denote by $\mathbb S_n$ the symmetric algebra generated by $1, x_1,\dots, x_n$. Then $\mathbb S_n$ is a $\mathbb W_n$-module under the action $\cdot$ given by applying the differential operators from $\mathbb W_n$ on the elements of $\mathbb S_n$, i.e.,
\[
\begin{array}{rcl}
\partial_i\cdot x_j&:=&\delta_{ij}\\
x_j\cdot x_i&:=&x_jx_i\quad .
\end{array}
\]

\section{Differential operator realizations of generalized Verma modules}\label{secTheoretical}
The statements in this section are contained in essence in \cite{Berline:algebrasDiffOps}. We present however our own constructive proofs, as they are the basis of our computer realization of the differential operator realizations of generalized Verma modules. 



\begin{lemma}\label{leCommuteBmaroundA} Let $a,b \in U({\mathfrak g})$. Then
\[
\begin{array}{rcl}
a b^{m}&=& \displaystyle\sum_{k=0}^{m} \binom{m}{k}(-1)^k b^{m-k}(\mathrm{ad} b)^k (a)\\
b^{m} a&=& \displaystyle\sum_{k=0}^{m} \binom{m}{k}(\mathrm{ad} b)^k (a)b^{m-k}
\quad .
\end{array}
\]
\end{lemma}
\begin{proof} (Dixmier)
Let $L_b$ and $R_b$ stand respectively for left and right multiplication by $b$ in $U({\mathfrak g})$. As $\mathrm{ad} b=L_b-R_b $ and $L_b$ and $R_b$ commute, the lemma follows from the Newton binomial formula.
\end{proof}
\begin{Definition}\label{defBcompatible}
Let $b_1,\dots, b_n$ be integers. We say that a polynomial $p(a_1,\dots, a_n)$ is $(b_1,\dots, b_n)$-compatible if  $p(a_1,\dots, a_n)$ is divisible by $a_k(a_k-1)\dots (a_k-b_k+1)$ whenever $b_k\geq 0$.
\end{Definition} 
The main motivation for the definition of $(b_1,\dots, b_n)$-compatible polynomials is the following lemma.
\begin{lemma}\label{leWriteAsDiffOperator}
Let $I$ be a finite indexing set. For each $i\in I$, let $(b_{i1}, \dots, b_{in} )$ be an $n$-tuple of integers  and $p_{i}(a_1,\dots, a_n)$ be a $(b_{i1}, \dots, b_{in} )$-compatible polynomial (Definition \ref{defBcompatible}). Suppose $g:\mathbb S_n\to\mathbb S_n$ is a linear operator such that
\[
g(x_1^{a_1}\dots x_n^{a_n})= \sum_{i} p_i(a_1, \dots, a_n)x_1^{a_1-b_{i1}} \dots x_{n}^{a_n-b_{in}}\quad ,
\]
for all monomials $x_1^{a_1}\dots x_n^{a_n}$. Then there exists an unique element $\omega\in \mathbb W_n$ such that the action of $g$ on $\mathbb S_n$ equals the $\cdot$-action of $\omega$ on $\mathbb S_n$. 
\end{lemma}
\begin{proof}
By linearity, it suffices to prove the Lemma in the case where $I$ has only one element, i.e.,  
\[ g(x_1^{a_1}\dots x_n^{a_n})= p(a_1, \dots, a_n)x_{1}^{a_1-b_1}\dots x_{n}^{a_n-b_n}\quad .
\]

Let $\mathbb Q[a_1,\dots, a_n]$ denote the polynomials in the variables $a_i$. 

\noindent Case 1. Suppose first $b_k\leq 0$ for all $k$. Define the $\mathbb Q$-linear map 
\begin{equation}\label{eqCaseBkNonnegative}
\begin{array}{rcl}
\omega_{b_1,\dots, b_n}:\mathbb Q[a_1,\dots, a_n]&\to& \mathbb W_n \\
a_1^{m_1}\dots a_n^{m_n}&\stackrel{\omega_{b_1,\dots, b_n}}{\mapsto}&\displaystyle
\prod_{i=1}^n\left( x_i^{-b_i} (x_i\partial_i)^{m_i}\right) \quad.
\end{array}
\end{equation}
Direct check shows that $\omega_{b_1,\dots, b_n} (p(a_1,\dots, a_n))$ has the same action $\mathbb S_n$ as $g$.

\noindent Case 2. Suppose $b_k>0$ for some $k$. By the requirements of the lemma the expression
\begin{equation}\label{eqPolyDivForWeyl}
p'(a_1,\dots, a_n):= \frac{p(a_1,\dots, a_n)}{\displaystyle\prod_{b_k>0} a_k(a_k-1)\dots (a_k-b_k+1) }
\end{equation}
is a polynomial. Define 
\[
\bar b_k=\left\{\begin{array}{ll}b_k& \mathrm {if~}b_k\leq 0 \\
0 &\mathrm{otherwise} \end{array}\right.
\quad .
\]
Direct check shows that the operator 
\begin{equation}\label{eqCaseBkNegative}
\omega:=\omega_{\bar b_1, \dots, \bar b_k}(p'(a_1,\dots, a_n)) \left(\prod_{b_k>0}\partial_k^{b_k}\right)
\end{equation}
has the same action on $\mathbb S_n$ as $g$. The uniqueness of $\omega$ follows from the fact that $\omega$ is uniquely determined by its action on monomials in the  $x_i$'s.
\end{proof}

Given  data  $(\{(p_i, b_{11}, \dots, b_{1n} )\}_{i\in I} )$ as in Lemma \ref{leWriteAsDiffOperator}, let 
\begin{equation}\label{eqNotationForDiffOperator}
\omega(\{(p_i, b_{11}, \dots, b_{1n} )\}_{i\in I} )
\end{equation}
denote the operator constructed in Lemma  \ref{leWriteAsDiffOperator}.

We will need the following observation.
\begin{lemma}\label{leCompatibilityTranslation}
If $p(a_1,\dots, a_n)$ is $(b_1,\dots, b_n)$-compatible, then $\binom{a_j}{c_j} p(a_1,\dots, a_j-c_j,\dots, a_n)$ is $(b_1,\dots,b_j+c_j,\dots, b_n)$-compatible.
\end{lemma}

Let $u_1,\dots, u_n $ be a basis of ${\mathfrak n}_-$ consisting of Chevalley-Weyl generators. We assume there is an order $<$ on these generators compatible with their labeling, that is,
\begin{equation}\label{eqPartialOrderChevGens}
u_1<u_2<\dots <u_n\quad .
\end{equation}
We pose no restrictions on $<$; in particular the order need not be compatible with any particular order of the weights of the generators. Changing the order causes the algorithm of the present paper to produce different differential operator realizations of the simple Lie algebras. The explicit form of the differential operators given in Table \ref{tableTheTable} depends on $<$.

For each generator $u_i$, let $\beta_i$ be the root corresponding to $u_i$. 

\begin{lemma}\label{leAdjointActionIsCompatible}
Let $g\in {\mathfrak g}$ and $j\in \{0,\dots, n\}$. There exists a finite indexing set $I(g)$, such that for each $i\in I(g)$ there exist integers $(b_{i1},\dots, b_{in})$, a $(b_{i1},\dots, b_{in})$-compatible polynomial $p_i(a_1, \dots, a_n)$ (Definition \ref{defBcompatible}) and an element $f_i\in U({\mathfrak g})$, such that $f_i$ is a Chevalley-Weyl generator lying outside of ${\mathfrak n}_-$ or equal to 1 and such that 
\begin{equation}\label{eqCommutingByInduction}
u_1^{a_1}\dots u_{j}^{a_j} g u_{j+1}^{a_{j+1}} u_n^{a_n}= \sum_{i\in I(g)} p_i (a_1,\dots, a_n)u_1^{a_1-b_{i1}}\dots u_n^{a_n-b_{in}}f_i\quad .
\end{equation}
\end{lemma}
\begin{proof}
It suffices to prove our claim in the case when $g$ is a Chevalley-Weyl generator. We proceed by double induction on the weight of $g$ and on the index $j$. We suppose we have proved by induction hypothesis our statement for all weights smaller than the weight of $g$. We suppose by second induction hypothesis that for all indices larger than $j$ we have proved our statement for weights smaller than or equal to the weight of $g$. We have three cases for the induction step. 

The first case is that $g=u_{j+1}$ or $g=u_j$; this case is trivial and there is nothing to prove. 

The second case is that $g=u_l$ with $l<j$. We apply Lemma \ref{leCommuteBmaroundA} with respect to $u_{j}^{a_j} g$ to obtain that our starting expression is a linear combination of $u_1^{a_1}\dots u_{j-1}^{a_{j-1}} u_l u_{j}^{a_{j}} \dots u_n^{a_n}$ and monomials of the form 
\[
\binom{a_j}{ c_j} u_1^{a_1}\dots u_{j-1}^{a_{j-1}} f u_{j}^{a_{j}-c_j}\dots  u_n^{a_n}
\]
with the weight of $f$ strictly smaller than the weight of $u_l$. By induction hypothesis these monomials can be written in the form \eqref{eqCommutingByInduction}. Here, the compatibility of the coefficients follows from Lemma \ref{leCompatibilityTranslation}. Continuing in this fashion, we may commute $u_l$ past $u_{j-1}^{a_{j-1}}$, and so on, until $u_l$ is positioned next to $u_{l}^{a_l}$ and so our induction step holds in the second case as well.

The final case is that $g$ is a Chevalley-Weyl generator that is not among the generators $u_{1},\dots,u_{j+1}$. We apply Lemma \ref{leCommuteBmaroundA} with respect to $g u_{j+1}^{a_{j+1}} $ to get that our starting expression is a linear combination of monomials of the form 
\[
\binom{a_{j+1}}{ c_{j+1}} u_1^{a_1}\dots u_{j+1}^{a_{j+1}-c_{j+1}} f u_{j+2}^{a_{j+2}} \dots u_n^{a_n}\quad ,
\]
where $f$ is a Chevalley-Weyl generator that is either equal to $g$ or is of strictly smaller weight. In both cases the induction step follows from the induction hypothesis and Lemma \ref{leCompatibilityTranslation}.
\end{proof}

Let $m_1,\dots, m_{\dim V_{\lambda}({\mathfrak l})}$ be a basis of $V_{\lambda}({\mathfrak l})$. By (\ref{eqGenVermaIsUnTimesV}) and the Poincare-Birkhoff-Witt theorem we can write every element in $ M_{\lambda}({\mathfrak g}, {\mathfrak p})$ uniquely as a linear combination of monomials the form 
\[
u_1^{a_1} \dots u_n^{a_n}\otimes_{U({\mathfrak p})} m_i\quad. 
\]

Define the linear map
\begin{equation}\label{eqUgonIsoToSn}
\begin{array}{rcl}
\varphi:U({\mathfrak n}_-) &\to&  \mathbb S_n\\
u_1^{a_1} \dots u_n^{a_n}&\stackrel{\varphi}{\mapsto}& x_1^{a_1}\dots x_n^{a_n}\quad .
\end{array}
\end{equation}

Define the linear map 
\[\begin{array}{rcl}
\varphi_\lambda: M_{\lambda}({\mathfrak g},{\mathfrak p}) &\to&   \mathbb S_n\otimes V_\lambda({\mathfrak l})\\
u_1^{a_1} \dots u_n^{a_n}\otimes_{U({\mathfrak p})} m_i&\stackrel{\varphi_\lambda}{\mapsto}&\varphi (u_1^{a_1} \dots u_n^{a_n})\otimes m_i =x_1^{a_1}\dots x_n^{a_n}\otimes m_i\quad .
\end{array}
\]
By (\ref{eqGenVermaIsUnTimesV}) the map $\varphi_\lambda$ is a vector space isomorphism between  $M_{\lambda}({\mathfrak g},{\mathfrak p})$ and $\mathbb S_n \otimes V_\lambda({\mathfrak l})$. We can therefore identify $\End \left(M_\lambda({\mathfrak g},{\mathfrak p})\right)$ with $\End (\mathbb S_n \otimes V_\lambda({\mathfrak l}))$ via the map $\varphi_\lambda'$ given by
\begin{equation}\label{eqVarphiPrime}
\begin{array}{rcl}
\varphi_\lambda':\End \left(M_\lambda({\mathfrak g},{\mathfrak p})\right)&\to& \End (\mathbb S_n \otimes V_\lambda({\mathfrak l}))\\
a &\stackrel{\varphi_\lambda'}{\mapsto}& \varphi_\lambda \circ a\circ \varphi_\lambda^{-1}\quad .
\end{array}
\end{equation}
We are now in a position to prove the following.
\begin{proposition}\label{propEmbedInWeylNoTensor}
Let ${\mathfrak g}$ be a semisimple Lie algebra with parabolic subalgebra ${\mathfrak p}$.
\begin{itemize}
\item[(a)] There exists a linear map  $\Phi_0:U({\mathfrak g})\to \mathbb W_n \otimes \id $ which makes the following diagram commutative.
\[
\xymatrix{
\End M_0({\mathfrak g}, {\mathfrak p}) \ar[r]^{\varphi'_0} &  \End \mathbb S_n \otimes \id&\simeq &\mathbb S_n\\
U({\mathfrak g}) \ar[u]^{\cdot} \ar[r]^{\Phi_0} & \mathbb W_n\otimes \id\ar[u]^{\cdot}&{\simeq} &\mathbb W_n
}
\quad .
\]
\item[(b)] $\Phi_0$ is a Lie algebra homomorphism.
\end{itemize}
\end{proposition}
\begin{proof}
(a) Let $g\in {\mathfrak g} $. By Lemma \ref{leAdjointActionIsCompatible}, there exists data as described in the Lemma such that \eqref{eqCommutingByInduction} holds. For each $g\in {\mathfrak g}$, let $I(g)$ be the set given by Lemma \ref{leAdjointActionIsCompatible}, and let 
\[
J(g):=\{i\in I| f_i=1\}
\] 
i.e., the subset corresponding to the summands in \eqref{eqCommutingByInduction} that end on 1. Using Lemma \ref{leAdjointActionIsCompatible} we can define the element $\omega(\{(p_i, b_{i1},\dots, b_{in})\}_{i\in J(g)})$ via \eqref{eqNotationForDiffOperator}, and therefore we can define the linear map $\Phi_0:{\mathfrak g} \to \mathbb W_n\otimes \id$  by 
\[
\Phi_0(g)  =\omega(\{(p_i, b_{i1},\dots, b_{in})\}_{i\in J(g)})\otimes \id\quad .
\]
Let $m_1$ be a non-zero vector of the one-dimensional vector space $V_0({\mathfrak l})$. The action of $\varphi_0'(g)$ on $\mathbb S_n\otimes m_1$ is computed on each monomial $x_1^{a_1}\dots x_n^{a_n}\otimes m_1$ by applying the map $\varphi_0$ to $gu_1^{a_1}\dots u_n^{a_j} \cdot m_1$. We have that the Chevalley-Weyl generators lying outside of ${\mathfrak n}$ have zero action on $m_1$. Together with Lemma \ref{leAdjointActionIsCompatible}, this implies that
\begin{equation*}\label{eqComputationMainProp}
\begin{array}{rcl}
\varphi_0\left(gu_1^{a_1}\dots u_n^{a_j} \cdot m_1 \right)&=&\varphi_0\left(\sum_{i\in I(g)} p_i (a_1,\dots, a_n)u_1^{a_1-b_{i1}}\dots u_n^{a_n-b_{in}}f_i\cdot m_1 \right)\\
&=&\varphi_0\left(\sum_{i\in J(g)} p_i (a_1,\dots, a_n)u_1^{a_1-b_{i1}}\dots u_n^{a_n-b_{in}}\cdot m_1\right)\\
&=&\sum_{i\in J(g)} p_i (a_1,\dots, a_n)x_1^{a_1-b_{i1}}\dots x_n^{a_n-b_{in}}\otimes m_1\\
&=&\Phi_0(g)\cdot (x_1^{a_1}\dots x_n^{a_n}\otimes m_1)\\
 &=&\Phi_0(g)\cdot \varphi_0(u_1^{a_1}\dots u_n^{a_n}\otimes_{U({\mathfrak p})} m_1) \quad .
\end{array}
\end{equation*}
The above proves (a).

\noindent(b) As $M_0({\mathfrak g},{\mathfrak p})$ is a (Lie algebra) representation of ${\mathfrak g}$, the image of ${\mathfrak g}$ in $\End M_0({\mathfrak g},{\mathfrak p})$ is a Lie algebra, and by (a) so is the image of $ {\mathfrak g}$ in $\End \mathbb S_n \otimes \id$.	
\end{proof}	

We finish this section with an interpretation of the action of ${\mathfrak g}$ as a Lie subalgebra of $\mathbb W_n\otimes \End V_\lambda({\mathfrak l})$. We recall that, given associative (and therefore Lie) algebras $A_1, A_2$, we have that $A_1\otimes A_2$ is an associative (and therefore Lie) algebra with $(a_1\otimes a_2)(b_1\otimes b_2)=a_1b_1\otimes a_2b_2$. Further, given two $A_i$-modules $M_i$, $i=1,2$, then $M_1\otimes M_2$ can be equipped with a $A_1\otimes A_2$-module structure via
\begin{equation}\label{eqDirectSumAction}
(a_1\otimes a_2) \cdot (m_1\otimes m_2):= (a_1\cdot m_1)\otimes( a_2\cdot m_2)\quad .
\end{equation} 
In particular, $\mathbb S_n\otimes V_\lambda({\mathfrak l})$ is a $\mathbb W_n\otimes\End V_\lambda({\mathfrak l})$-module; the associative algebra $\mathbb W_n\otimes\End V_\lambda({\mathfrak l})$ is sometimes called a ``Weyl algebra with matrix coefficients''.
\begin{proposition}\label{propGenVermaIsInWeylPlusEnd}
There exists a Lie algebra homomorphism  $\Phi_\lambda:U({\mathfrak g})\to \mathbb W_n\otimes \End V_\lambda({\mathfrak l})  $ which makes the following diagram commutative.
\[
\xymatrix{
\End M_\lambda({\mathfrak g}, {\mathfrak p}) \ar[r]^{\varphi'_\lambda} & \End\left(\mathbb S_n\otimes V_\lambda({\mathfrak l})\right) \\
U({\mathfrak g}) \ar[u]^{\cdot} \ar[r]^{\Phi_\lambda} & \mathbb W_n\otimes\End V_\lambda({\mathfrak l}) \ar[u]^{\eqref{eqDirectSumAction}}
}
\quad ,
\]
where we are using \eqref{eqDirectSumAction} with respect to the $\cdot$ action of $\mathbb W_n$ on $\mathbb S_n$.
\end{proposition}
\textbf{Remark.} The proposition remains valid if we replace $V_\lambda({\mathfrak l})$ with an arbitrary ${\mathfrak p}$-module $V$ and $\varphi'_\lambda$ by an operator defined similarly to \eqref{eqVarphiPrime}.

\begin{proof}
Let $g\in {\mathfrak g}$. Let $\{(f_i,p_i,b_{i1},\dots, b_{in})\}_{i\in I}$ be the data whose existence is given by Lemma \ref{leCompatibilityTranslation}. We can therefore define $\Phi_\lambda: U({\mathfrak g})\to \mathbb W_n\otimes\End V_\lambda({\mathfrak l})$ using Lemma \ref{leCompatibilityTranslation} so that the last equality in the following computation holds for an arbitrary vector $v\in V_\lambda({\mathfrak l})$.
\[
{\renewcommand{\arraystretch}{1.5}
\begin{array}{rcl}
\varphi_\lambda\left(gu_1^{a_1}\dots u_n^{a_j} \cdot v \right)&=&\displaystyle \varphi_\lambda\left(\sum_{i\in I(g)} p_i (a_1,\dots, a_n)u_1^{a_1-b_{i1}}\dots u_n^{a_n-b_{in}}\otimes_{U({\mathfrak p})} (f_i \cdot v) \right)\\
&=&\displaystyle \sum_{i\in I(g)} p_i (a_1,\dots, a_n)x_1^{a_1-b_{i1}}\dots x_n^{a_n-b_{in}}\otimes (f_i \cdot v) \\
&=:&\displaystyle \Phi_\varphi(g) (x_1^{a_1}\dots x_n^{a_n}\otimes v)\quad .
\end{array}
}
\]
The remainder of the proof is analogous to that of Proposition \ref{propEmbedInWeylNoTensor} and we omit it.
\end{proof}
\section{An informal algorithm description}\label{secComputerWork}

In the present section we describe an algorithm that explicitly computes the homomorphism $\Phi_\lambda$ from Proposition \ref{propGenVermaIsInWeylPlusEnd}. The homomorphism is described by giving the image of each of the simple Chevalley-Weyl generators of ${\mathfrak g}$.

For a simple complex Lie algebra ${\mathfrak g}$, the  only input data needed to compute an embedding ${\mathfrak g} \hookrightarrow \mathbb W_n\otimes \End V_\lambda({\mathfrak l})$ is the data ${\mathfrak g},\lambda, {\mathfrak p}$. More precisely, ${\mathfrak g}$ can be described to a computer program by the type of the Lie algebra, $\lambda$ by a vector given as the sequence of its fundamental coordinates and ${\mathfrak p}$ by a vector of zeroes and ones indicating which simple roots are crossed out. 

The root systems of each simple Lie algebra can be generated, for example, using \cite[page 162]{DeGraaf:LieAlgebrasAndAlgorithms}. The structure constants of a semisimple Lie algebra relative to a Chevalley-Weyl basis (in particular, their signs), can be computed with a uniform method in both the simply and non-simply laced case according to \cite[pages 47-51]{Samelson} as implemented in our C++ program, \cite{Milev:vpf}.

We suppose $V_\lambda({\mathfrak l})$ is constructed, in particular we have a vector space basis $m_1,\dots, m_{\dim V_{\lambda}({\mathfrak l})}$ of $V_\lambda({\mathfrak l})$ and we know the action of each $g\in {\mathfrak g}$  as an element of $ V_\lambda({\mathfrak l})\otimes V_\lambda({\mathfrak l})^* \simeq  \End V_\lambda({\mathfrak l}) $. An algorithm for  computing the actions of $g$ on $V_\lambda({\mathfrak l})$ from the data ${\mathfrak g},\lambda, {\mathfrak p}$ can be extracted from \cite{Littelmann:ConesCrystalsPatterns} and \cite{Littelmann:PathsAndRootOperators}; a detailed algorithm description is in preparation, \cite{JacksonMilev:ShapovalovForm}. The algorithm is already implemented and tested in \cite{Milev:vpf}.

Let $g_i$ denote the Chevalley-Weyl generators computed using the preceding algorithms. We represent (non-uniquely) a monomial  
\begin{equation*}\label{eqGenericMonomial}
 g_{i_1}^{p_{1}}\dots g_{i_k}^{p_{k}}
\end{equation*}
in $U({\mathfrak g})$ as a list of generator indices and a corresponding list of exponents, i.e., as the \emph{monomial data}
\begin{equation}\label{eqMonData}
((i_1,\dots, i_k),(p_1, \dots, p_k))\quad i_s\in \mathbb Z, p_k\mathrm{~lie~in~the~base~ring}.
\end{equation}
We allow the exponents to belong to any base ring defined in our program that contains\footnote{to achieve this we use C++ templates} $\mathbb Z$. We represent (non-uniquely) an arbitrary element 
\[
\sum a_{(i_1,\dots i_k), (p_1, \dots, p_k)} g_{i_1}^{p_{1}}\dots g_{i_k}^{p_{k}}
\]
in $U({\mathfrak g})$ as a collection of monomial data \eqref{eqMonData} and a list of coefficients lying in the base ring, i.e., as the \emph{element data}
\begin{equation}\label{eqEltData}
((q_1,\dots, q_s), (\textbf{mon}_1, \dots, \textbf{mon}_s))\quad \textbf{mon}_t\mathrm{~of~the~form~}\eqref{eqMonData},~q_t\mathrm{~lie~in~the~base~ring}.
\end{equation}
\onlineVersionOnly{
In the above data, we do not allow a monomial $\textbf{mon}_t$ to appear more than once, and we do not allow $q_t=0$. To the data above, we may add a monomial $\textbf{mon}_{s'}$ together with a coefficient $q_{s'}$ by the following procedure. If the monomial $\textbf{mon}_{s'}$ is not already contained in the list of monomials, we add $\textbf{mon}_{s'}$ and $q_{s'}$ at the end of the corresponding lists in \eqref{eqEltData}. Otherwise, we find the index $t$ for which $\textbf{mon}_t=\textbf{mon}_{s'}$ and add $q_{s'}$ to $q_t$; if the resulting coefficient $q_{s'}+q_t$ is non-zero we are done, else we remove $\textbf{mon}_t$ and the corresponding coefficient from the data \eqref{eqEltData}.
}

\subsection{An iterative procedure implementing Lemma \ref{leAdjointActionIsCompatible}}
Let $<$ be a total order on the Chevalley-Weyl generators $g_i$. For the time being we assume this order to be arbitrary; later on we will fix $<$ to be an extension of the order of the generators of ${\mathfrak n}_-$ given in \eqref{eqPartialOrderChevGens}. 

\begin{Definition}\label{defDataReducedWRTorder}
Suppose that the monomial data \eqref{eqMonData} giving the monomial $g_{i_1}^{p_{1}}$ $\dots$ $g_{i_k}^{p_{k}}$ has the following property: if there is a generator $g_{i_1}^{p_1}$ that appears to the left of a generator $g_{i_2}^{p_2}$ and at the same time $g_{i_1}>g_{i_2}$, then both $p_1$ and $p_2$ are not non-negative integers. Then we say that the monomial data \eqref{eqMonData} is reduced relative to $<$. 

Suppose each monomial data in the element data \eqref{eqEltData} is reduced relative to $<$. Then we say that the element data \eqref{eqEltData} is reduced relative to $<$. 
\end{Definition}

Let  $u\in U({\mathfrak g})$. Suppose $u$ is represented by element data $X$ which is not reduced relative to $<$. We introduce an iterative procedure that, after consecutive application, aims to transform $X$ to element data $X'$ that is reduced relative to $<$ and represents the same element. The procedure is in essence an implementation of the proof of Lemma \ref{leAdjointActionIsCompatible}.

We treat each monomial data $\textbf{mon}$ participating in \eqref{eqEltData} separately.  Suppose $\textbf{mon}$ is not reduced relative to $<$. We remove $\textbf{mon}$ from the data \eqref{eqEltData}. Then, we transform $\textbf{mon}$ so no two consecutive generator indices $i_j, i_{j+1}$ are equal by adding the corresponding exponents and removing one of the indices. If after this transformation $\textbf{mon}$ is reduced relative to $<$ we add it back to the data \eqref{eqEltData}. Otherwise, we find the first generator index $i_j$ for which $g_{i_j}>g_{i_{j+1}} $ and for which either
\begin{itemize}
\item $p_j$ is a positive integer and $g_{i_{j+1}}$ is nilpotent or
\item $p_{j+1}$ is a positive integer and $g_{i_j}$ is nilpotent.
\end{itemize}
We then apply Lemma \ref{leCommuteBmaroundA}, choosing the first identity if  $p_{j}$ is a positive integer and the second identity otherwise. The right hand side of the identities in Lemma \ref{leCommuteBmaroundA} provides us with new set monomials $\textbf{mon}_1,\dots, \textbf{mon}_k$ and a new set of coefficients; we multiply those coefficients by the initial coefficient of $\textbf{mon}$, and add back to the data \eqref{eqEltData}.

\subsection{Algorithm description}\label{secAlgorithmDescription}
We are ready to implement Proposition \ref{propGenVermaIsInWeylPlusEnd}. We recall that we have a fixed semisimple Lie algebra ${\mathfrak g}$, a highest weight $\lambda$, and a nilradical ${\mathfrak n}_-$. We recall from Section \ref{secTheoretical} that $u_1, \dots, u_n$ are the Chevalley-Weyl generators of ${\mathfrak n}_-$. 

From now on fix the order $<$ so that $u_i<g_j$ for any Chevalley-Weyl generator $g_j\notin {\mathfrak n}_-$ and so that the order is an extension of  \eqref{eqPartialOrderChevGens}, that is $u_1<u_{2}<\dots <u_n $. Let $g_{\pm 1}, \dots g_{\pm k}$ be the simple Chevalley-Weyl generators of ${\mathfrak g}$. Consider the action of left multiplication by $g_{\pm i}$ on a generic monomial in $U({\mathfrak n}_-)$, i.e., consider the element
\begin{equation}\label{eqActionOnGenericEltUn-}
g_{\pm i}u_1^{a_1}\dots u_n^{a_n}\quad .
\end{equation}

The proof of Lemma \ref{leAdjointActionIsCompatible} shows that the consecutive application of the iterative procedure described after Definition \ref{defDataReducedWRTorder} transform the element data corresponding to \eqref{eqActionOnGenericEltUn-} to element data that is reduced relative to $<$. The so computed data represents the sum \eqref{eqCommutingByInduction} from  the proof of Lemma \ref{leAdjointActionIsCompatible}. In particular, we can read off the set of integers $(b_{i1},\dots, b_{in})$ and $(b_{i1},\dots, b_{in})$-compatible polynomials $p_i(a_1, \dots, a_n)$ whose existence is asserted by Lemma \ref{leAdjointActionIsCompatible}. 

Fix now the base ring to be $\mathbb Q[a_1,\dots, a_n]$. \onlineVersionOnly{We note that our final implementation allows computations over the ring $\mathbb Q[a_1,\dots, a_n][\lambda_1,\dots, \lambda_s]$, where $\lambda_i$ are indeterminates that correspond to the crossed-out roots of ${\mathfrak p}$ (we recall that $\lambda$ need only be integral with respect to the simple roots of ${\mathfrak l}$, in particular $\lambda$ is allowed to have indeterminate entries in the coordinates corresponding to crossed-out roots).} We recall that the monomials participating in \eqref{eqCommutingByInduction} are of the form 
\[ p (a_1,\dots, a_n)u_1^{a_1-b_{1}}\dots u_n^{a_n-b_{n}}f\quad ,
\]
where $f$ is some Chevalley-Weyl generator of ${\mathfrak g}$ lying outside of ${\mathfrak n}_-$ and the polynomial $p$ is $(b_{1}, \dots, b_n)$-compatible (Definition \ref{defBcompatible}). First, we apply Lemma \ref{leWriteAsDiffOperator} to the first multiplicand $p (a_1,\dots, a_n)u_1^{a_1-b_{1}}\dots u_n^{a_n-b_{n}}$  - for \eqref{eqPolyDivForWeyl} we use polynomial division, and for the remainder of the Lemma we use a standard implementation of $\mathbb W_n$. If $\lambda=0$, we replace $f$ by $\id$; else we replace $f$ by its action on $V_{\lambda}({\mathfrak l})$ and take the tensor product with the differential operator computed in the preceding step. We collect all summands to complete the computation of the homomorphism $\Phi_\lambda$ from Proposition \ref{propGenVermaIsInWeylPlusEnd}. 

\subsection{Proof of main theorem and implementation comments}
To apply the algorithm in Section \ref{secAlgorithmDescription} we fix an order $<$ on the Chevalley generators of ${\mathfrak n}_-$ by declaring that $u_1<u_2$ if the weight $\beta_1$ of $u_1$ is less than the weight $\beta_2$ of $u_2$, where we are comparing $\beta_1$ and $\beta_2$ in the graded lexicographic order in with respect to their simple basis coordinates.

\begin{proof}[~of Theorem \ref{thExceptionalsinWn}] We run the algorithm described in Section \ref{secAlgorithmDescription} as implemented in \cite{Milev:vpf} to generate Table \ref{tableTheTable}. 
\end{proof}

We are including the number of rational number arithmetic operations made by our program during the computation of Table \ref{tableTheTable}. We are also including computation times, however those are not indicative as we have not made any special attempts to optimize our code for speed. All computations were carried out on a low-end 32-bit laptop computer (1.6 Ghz). 

\noindent\begin{tabular}{|cccc|}
\hline
${\mathfrak g}$ & ${\mathfrak p}$& \begin{tabular}{c} Number of arithmetic \\operations need to \\ generate the embeddings \end{tabular} & Running time \\ \hline
$E_8$& last root crossed  & 28,754,380 & 52.10 seconds\\
$E_7$ & last root crossed & 2,431,419  & 4.50 seconds\\
$E_6$ & last root crossed & 487,021  & 0.99 seconds\\
$F_4$& first (long) root  crossed & 374,377& 0.82 seconds\\
$F_4$& last (short) root crossed &469,892 &1.07 seconds\\
$G_2$& first (short) root crossed & 22,185 & 0.06 seconds \\
$G_2$& last (long) root crossed &  14,072 & 0.04 seconds\\
\hline
\end{tabular}

\onlineVersionOnly{
Table \ref{tableTheTable} was generated using the following commands of the on-line interface of our program \cite{Milev:vpf}.

\noindent \href{http://cartan.math.umb.edu/vpf/cgi-bin/calculator?textInput=MapSemisimpleLieAlgebraInWeylAlgebra\%7B\%7D\%28G\_2\%2C+\%280\%2C0\%29\%2C+\%28+1\%2C+0\%29\%29\%3B+&buttonGo=Go}{MapSemisimpleLieAlgebraInWeylAlgebra(G\_2,(0,0),( 1, 0));}

\noindent \href{http://cartan.math.umb.edu/vpf/cgi-bin/calculator?textInput=MapSemisimpleLieAlgebraInWeylAlgebra\%7B\%7D\%28G\_2\%2C+\%280\%2C0\%29\%2C+\%28+0\%2C+1\%29\%29\%3B+}{MapSemisimpleLieAlgebraInWeylAlgebra(G\_2,(0,0),(0,1));}

\noindent \href{http://cartan.math.umb.edu/vpf/cgi-bin/calculator?textInput=MapSemisimpleLieAlgebraInWeylAlgebra\%7B\%7D\%28F\_4\%2C+\%280\%2C0\%2C0\%2C0\%29\%2C+\%28+0\%2C0\%2C0\%2C+1\%29\%29\%3B+&buttonGo=Go}{MapSemisimpleLieAlgebraInWeylAlgebra(F\_4,(0,0,0,0),(0,0,0,1));}

\noindent \href{http://cartan.math.umb.edu/vpf/cgi-bin/calculator?textInput=MapSemisimpleLieAlgebraInWeylAlgebra\%7B\%7D\%28F\_4\%2C+\%280\%2C0\%2C0\%2C0\%29\%2C+\%28+1\%2C0\%2C0\%2C+0\%29\%29\%3B+}{MapSemisimpleLieAlgebraInWeylAlgebra(F\_4, (0,0,0,0), ( 1,0,0, 0)); }

\noindent \href{http://cartan.math.umb.edu/vpf/cgi-bin/calculator?textInput=MapSemisimpleLieAlgebraInWeylAlgebra\%7B\%7D\%28E\_6\%2C+\%280\%2C0\%2C0\%2C0\%2C0\%2C0\%29\%2C+\%28+1\%2C0\%2C0\%2C0\%2C0\%2C+0\%29\%29\%3B+}{MapSemisimpleLieAlgebraInWeylAlgebra(E\_6, (0,0,0,0,0,0), ( 1,0,0,0,0, 0)); }

\noindent \href{http://cartan.math.umb.edu/vpf/cgi-bin/calculator?textInput=MapSemisimpleLieAlgebraInWeylAlgebra\%7B\%7D\%28E\_7\%2C+\%280\%2C0\%2C0\%2C0\%2C0\%2C0\%2C0\%29\%2C+\%28+0\%2C0\%2C0\%2C0\%2C0\%2C+0\%2C1\%29\%29\%3B+&buttonGo=Go}{MapSemisimpleLieAlgebraInWeylAlgebra(E\_7, (0,0,0,0,0,0,0), ( 0,0,0,0,0, 0,1)); }

\noindent
MapSemisimpleLieAlgebraInWeylAlgebra\{\}(E\_8, (0,0,0,0,0,0,0,0), ( 0,0,0,0,0, 0,0,1)); 

}
\onlineVersionOnly{
\noindent We note that in our C++ implementation, we represent elements of $ U({\mathfrak g})$, polynomials, elements of $\mathbb W_n$, elements of $\simeq V_\lambda({\mathfrak l})\otimes V_\lambda({\mathfrak l})^* \simeq  \End V_\lambda({\mathfrak l}) $ and elements of $ \mathbb W_n\otimes \End V_{\lambda}({\mathfrak l}) $ using a single C++  monomial collections class. This allows us to reuse all linear algebra and memory management routines (hash tables, etc.). The monomial collections that have the property that a product of two monomials is a monomial use in addition shared code for multiplication\footnote{This is again achieved using C++ templates}.
}

We finish this section by noting there is a simple consistency check for our computer computations. Let ${\mathfrak g}'\subset\mathbb W_n\otimes \End V_\lambda({\mathfrak l})$ be the Lie algebra generated by the differential operators corresponding to simple generators as constructed above. As asserted by Proposition \ref{propGenVermaIsInWeylPlusEnd}, we have that ${\mathfrak g}\simeq{\mathfrak g}'$. In particular we can check that $\dim {\mathfrak g}'=\dim{\mathfrak g}$: this makes sense as the computer code that computes the differential operator Lie algebra ${\mathfrak g}'$ is independent from the code that carries out our algorithm. In this way we are testing two independent computer programs against one another,  and the output of both programs against the already known dimension $\dim {\mathfrak g}$. All printouts in Table \ref{tableTheTable} have passed such a consistency check.

\section{Example: the embedding $G_2\hookrightarrow \mathbb W_5$}\label{secTheConstructionCaseG2}
In the present section we illustrate the computation of $G_2\hookrightarrow \mathbb W_5$, as reported by our computer program. A basis of $G_2$ is computed in Table \ref{tableGtwoStructure}. The elements $h_i$ are dual to the simple roots of $G_2$, and the elements $g_i$ together with $h_1, \frac{1}{3} h_2$ form a Chevalley-Weyl basis of $G_2$. The generators $g_i$ are indexed according to the index of the corresponding root space, and the root spaces are indexed according to their order in the simple coordinate graded lexicographic order.

Both proper parabolic subalgebras of $G_2$ are maximal. We carry out the computations for the parabolic subalgebra obtained by crossing out the short root. A basis for the nilradical ${\mathfrak n}_-$ is given by $g_{-6},g_{-5}, g_{-4}, g_{-3}, g_{-1}$. An arbitrary monomial in $U({\mathfrak n}_-)$ is then a multiple of 
\[
u:=g_{-6}^{a_{1}}g_{-5}^{a_{2}}g_{-4}^{a_{3}}g_{-3}^{a_{4}}g_{-1}^{a_{5}}\quad .
\] 
Applying consecutively the identities of Lemma \ref{leCommuteBmaroundA} according to Section \ref{secComputerWork}, we can compute that 
\[
\begin{array}{rcl}
g_1 u&=&-3a_{4}g_{-6}^{a_{1}}g_{-5}^{a_{2}}g_{-4}^{a_{3}}g_{-3}^{a_{4}-1}g_{-1}^{a_{5}}g_{-2}+a_{5}g_{-6}^{a_{1}}g_{-5}^{a_{2}}g_{-4}^{a_{3}}g_{-3}^{a_{4}}g_{-1}^{a_{5}-1}h_{1}\\
&&+g_{-6}^{a_{1}}g_{-5}^{a_{2}}g_{-4}^{a_{3}}g_{-3}^{a_{4}}g_{-1}^{a_{5}}g_{1}  +(-3a_{3}^{2}+3a_{3})g_{-6}^{a_{1}+1}g_{-5}^{a_{2}}g_{-4}^{a_{3}-2}g_{-3}^{a_{4}}g_{-1}^{a_{5}} \\
&&+
(9a_{4}^{2}a_{5}^{2}-9a_{4}a_{5}^{2}-9a_{4}^{2}a_{5}+9a_{4}a_{5})g_{-6}^{a_{1}+1}g_{-5}^{a_{2}}g_{-4}^{a_{3}}g_{-3}^{a_{4}-2}g_{-1}^{a_{5}-2} \\
&&-a_{2}g_{-6}^{a_{1}}g_{-5}^{a_{2}-1}g_{-4}^{a_{3}+1}g_{-3}^{a_{4}}g_{-1}^{a_{5}}\\
&&+(-3a_{4}a_{5}^{3}+9a_{4}a_{5}^{2}-6a_{4}a_{5})g_{-6}^{a_{1}}g_{-5}^{a_{2}+1}g_{-4}^{a_{3}}g_{-3}^{a_{4}-1}g_{-1}^{a_{5}-3} \\&& -2a_{3}g_{-6}^{a_{1}}g_{-5}^{a_{2}}g_{-4}^{a_{3}-1}g_{-3}^{a_{4}+1}g_{-1}^{a_{5}}\\&&
+(3a_{4}a_{5}^{2}-3a_{4}a_{5})g_{-6}^{a_{1}}g_{-5}^{a_{2}}g_{-4}^{a_{3}+1}g_{-3}^{a_{4}-1}g_{-1}^{a_{5}-2} \\&& +(-a_{5}^{2}-3a_{4}a_{5}+a_{5})g_{-6}^{a_{1}}g_{-5}^{a_{2}}g_{-4}^{a_{3}}g_{-3}^{a_{4}}g_{-1}^{a_{5}-1}\\
g_{-1}u&=& (-3a_{4}^{2}+3a_{4})g_{-6}^{a_{1}+1}g_{-5}^{a_{2}}g_{-4}^{a_{3}}g_{-3}^{a_{4}-2}g_{-1}^{a_{5}}-3a_{3}g_{-6}^{a_{1}}g_{-5}^{a_{2}+1}g_{-4}^{a_{3}-1}g_{-3}^{a_{4}}g_{-1}^{a_{5}}\\&&-2a_{4}g_{-6}^{a_{1}}g_{-5}^{a_{2}}g_{-4}^{a_{3}+1}g_{-3}^{a_{4}-1}g_{-1}^{a_{5}}  +g_{-6}^{a_{1}}g_{-5}^{a_{2}}g_{-4}^{a_{3}}g_{-3}^{a_{4}}g_{-1}^{a_{5}+1}\\
g_2 u&=&
g_{-6}^{a_{1}}g_{-5}^{a_{2}}g_{-4}^{a_{3}}g_{-3}^{a_{4}}g_{-1}^{a_{5}}g_{2}-a_{1}g_{-6}^{a_{1}-1}g_{-5}^{a_{2}+1}g_{-4}^{a_{3}}g_{-3}^{a_{4}}g_{-1}^{a_{5}}\\&&
+(-2a_{4}^{3}+6a_{4}^{2}-4a_{4})g_{-6}^{a_{1}+1}g_{-5}^{a_{2}}g_{-4}^{a_{3}}g_{-3}^{a_{4}-3}g_{-1}^{a_{5}} \\&& +(-a_{4}^{2}+a_{4})g_{-6}^{a_{1}}g_{-5}^{a_{2}}g_{-4}^{a_{3}+1}g_{-3}^{a_{4}-2}g_{-1}^{a_{5}}+a_{4}g_{-6}^{a_{1}}g_{-5}^{a_{2}}g_{-4}^{a_{3}}g_{-3}^{a_{4}-1}g_{-1}^{a_{5}+1}\\
g_{-2}u&=&
g_{-6}^{a_{1}}g_{-5}^{a_{2}}g_{-4}^{a_{3}}g_{-3}^{a_{4}}g_{-1}^{a_{5}}g_{-2}-a_{2}g_{-6}^{a_{1}+1}g_{-5}^{a_{2}-1}g_{-4}^{a_{3}}g_{-3}^{a_{4}}g_{-1}^{a_{5}}\\&&
+(-3a_{4}a_{5}^{2}+3a_{4}a_{5})g_{-6}^{a_{1}+1}g_{-5}^{a_{2}}g_{-4}^{a_{3}}g_{-3}^{a_{4}-1}g_{-1}^{a_{5}-2} \\&& +(a_{5}^{3}-3a_{5}^{2}+2a_{5})g_{-6}^{a_{1}}g_{-5}^{a_{2}+1}g_{-4}^{a_{3}}g_{-3}^{a_{4}}g_{-1}^{a_{5}-3}\\&&
+(-a_{5}^{2}+a_{5})g_{-6}^{a_{1}}g_{-5}^{a_{2}}g_{-4}^{a_{3}+1}g_{-3}^{a_{4}}g_{-1}^{a_{5}-2} \\&& +a_{5}g_{-6}^{a_{1}}g_{-5}^{a_{2}}g_{-4}^{a_{3}}g_{-3}^{a_{4}+1}g_{-1}^{a_{5}-1}\quad .
\end{array}
\]
Applying the map $\varphi_0$ to the above expressions amounts to renaming the generators $g_{-6},g_{-5}, g_{-4}, g_{-3}, g_{-1}$ to $x_1,\dots, x_5$, replacing the remaining generators by 0, and finally tensoring on the right with $m_1$. We recall from Lemma \ref{leAdjointActionIsCompatible} that if one of the so obtained monomials is of the form $p(a_1, \dots, a_5)x_1^{a_1-b_1}\dots x_5^{a_5-b_5}$ and $b_i>0$, then $p(a_1, \dots, a_5)$ is divisible by $a_i(a_i-1)\dots (a_i-b_i+1)$. Now we can apply \eqref{eqCaseBkNegative} and \eqref{eqCaseBkNonnegative} from Lemma \ref{leWriteAsDiffOperator} to get that
\[
\begin{array}{rcl}
\Phi_0 (g_1)&=& (-3x_{2}\partial_{4}\partial_{5}^{3}+9x_{1}\partial_{4}^{2}\partial_{5}^{2}  +3x_{3}\partial_{4}\partial_{5}^{2}-x_{5}\partial_{5}^{2} \\&& -3x_{4}\partial_{4}\partial_{5}-3x_{1}\partial_{3}^{2}  -2x_{4}\partial_{3}-x_{3}\partial_{2})\otimes \id\\
\Phi_0 (g_{-1}) &=&(-3x_{1}\partial_{4}^{2}-2x_{3}\partial_{4} -3x_{2}\partial_{3}+x_{5})\otimes \id\\
\Phi_0 (g_2)&=&(-2x_{1}\partial_{4}^{3}-x_{3}\partial_{4}^{2} +x_{5}\partial_{4}-x_{2}\partial_{1})\otimes \id\\
\Phi_0(g_{-2})&=&(x_{2}\partial_{5}^{3}-3x_{1}\partial_{4}\partial_{5}^{2}  -x_{3}\partial_{5}^{2}+x_{4}\partial_{5}-x_{1}\partial_{2})\otimes \id\quad .
\end{array}
\]
\begin{landscape}

Table \ref{tableTheTable}. Elements of $\mathbb W_n$ generating the exceptional Lie algebras. Each entry of the table is a simple Chevalley-Weyl generator, given in the order $g_1, g_{-1}, g_{2}, g_{-2},\dots$. Here, $g_{\pm i}$ stands for the Chevalley-Weyl generator of the simple root whose $i^{th}$ coordinate is non-zero. The simple roots are ordered in the same order as the one implied by \cite[page 59]{Humphreys}. For $G_2$ and $F_4$ there are two entries as $G_2$ and $F_4$ each have two non-isomorphic parabolic subalgebras of maximal dimension. For $G_2, F_4$, the first entry in the table corresponds to first root being crossed out (the first root is short for $G_2$ and long for $F_4$). The other entry corresponds to long root being crossed out.
\begin{longtable}
{rll}\label{tableTheTable}
${\mathfrak g}$& $n$& \multicolumn{1}{c}{ element of $\mathbb W_n\simeq \mathbb W_n\otimes \id$} \\
\endhead

\hline\multirow{4}{*}{$G_2$} & \multirow{4}{*}{5}& 
$\begin{array}{l}(-3x_{2}\partial_{4}\partial_{5}^{3}+9x_{1}\partial_{4}^{2}\partial_{5}^{2}+3x_{3}\partial_{4}\partial_{5}^{2}-x_{5}\partial_{5}^{2}-3x_{4}\partial_{4}\partial_{5}-3x_{1}\partial_{3}^{2}-2x_{4}\partial_{3}-x_{3}\partial_{2})\otimes id\end{array}$\\\cline{3-3} 
&& $\begin{array}{l}(-3x_{1}\partial_{4}^{2}-2x_{3}\partial_{4}-3x_{2}\partial_{3}+x_{5})\otimes id\end{array}$\\\cline{3-3} 
&& $\begin{array}{l}(-2x_{1}\partial_{4}^{3}-x_{3}\partial_{4}^{2}+x_{5}\partial_{4}-x_{2}\partial_{1})\otimes id\end{array}$\\\cline{3-3} 
&& $\begin{array}{l}(x_{2}\partial_{5}^{3}-3x_{1}\partial_{4}\partial_{5}^{2}-x_{3}\partial_{5}^{2}+x_{4}\partial_{5}-x_{1}\partial_{2})\otimes id\end{array}$\\
\hline\multirow{4}{*}{ $G_2$} & \multirow{4}{*}{5}& $\begin{array}{l}(-3x_{5}\partial_{4}-3x_{1}\partial_{3}^{2}-2x_{4}\partial_{3}-x_{3}\partial_{2})\otimes id\end{array}$\\ \cline{3-3} 
&& $\begin{array}{l}(-x_{4}\partial_{5}-3x_{1}\partial_{4}^{2}-2x_{3}\partial_{4}-3x_{2}\partial_{3})\otimes id\end{array}$\\ \cline{3-3} 
&& $\begin{array}{l}(-x_{5}\partial_{5}^{2}-x_{4}\partial_{4}\partial_{5}-2x_{1}\partial_{4}^{3}-x_{3}\partial_{4}^{2}-x_{2}\partial_{1})\otimes id\end{array}$\\ \cline{3-3} 
&& $\begin{array}{l}(-x_{1}\partial_{2}+x_{5})\otimes id\end{array}$\\ \hline 
\hline\multirow{8}{*}{$F_4$} & \multirow{8}{*}{15}& $\begin{array}{l}(-x_{15}\partial_{15}^{2}-x_{14}\partial_{14}\partial_{15}-x_{13}\partial_{13}\partial_{15}-x_{12}\partial_{12}\partial_{15}-x_{11}\partial_{11}\partial_{15}+x_{1}\partial_{7}\partial_{10}\partial_{15}-x_{9}\partial_{9}\partial_{15}-x_{7}\partial_{7}\partial_{15}+x_{1}\partial_{7}\partial_{12}\partial_{14} \\ +x_{10}\partial_{12}\partial_{14}+2x_{1}\partial_{9}^{2}\partial_{14}+x_{8}\partial_{9}\partial_{14}+x_{5}\partial_{7}\partial_{14}+x_{1}\partial_{7}\partial_{13}^{2}+x_{10}\partial_{13}^{2}+4x_{1}\partial_{9}\partial_{11}\partial_{13}+x_{8}\partial_{11}\partial_{13}-x_{6}\partial_{9}\partial_{13}-x_{4}\partial_{7}\partial_{13} \\ -2x_{1}\partial_{11}^{2}\partial_{12}+x_{6}\partial_{11}\partial_{12}-x_{3}\partial_{7}\partial_{12}-x_{5}\partial_{11}^{2}-x_{4}\partial_{9}\partial_{11}-x_{2}\partial_{7}\partial_{10}-x_{3}\partial_{9}^{2}-x_{2}\partial_{1})\otimes id\end{array}$\\\cline{3-3} 
&& $\begin{array}{l}(-x_{1}\partial_{2}+x_{15})\otimes id\end{array}$\\\cline{3-3} 
&& $\begin{array}{l}(x_{15}\partial_{14}-x_{12}\partial_{10}-x_{1}\partial_{8}^{2}-x_{9}\partial_{8}-x_{7}\partial_{5}-x_{3}\partial_{2})\otimes id\end{array}$\\\cline{3-3} 
&& $\begin{array}{l}(x_{14}\partial_{15}-x_{10}\partial_{12}-x_{1}\partial_{9}^{2}-x_{8}\partial_{9}-x_{5}\partial_{7}-x_{2}\partial_{3})\otimes id\end{array}$\\\cline{3-3} 
&& $\begin{array}{l}(2x_{14}\partial_{13}-x_{13}\partial_{12}-x_{11}\partial_{9}-x_{8}\partial_{6}-2x_{5}\partial_{4}-x_{4}\partial_{3})\otimes id\end{array}$\\\cline{3-3} 
&& $\begin{array}{l}(x_{13}\partial_{14}-2x_{12}\partial_{13}-x_{9}\partial_{11}-x_{6}\partial_{8}-x_{4}\partial_{5}-2x_{3}\partial_{4})\otimes id\end{array}$\\\cline{3-3} 
&& $\begin{array}{l}(x_{13}\partial_{11}+2x_{12}\partial_{9}-2x_{1}\partial_{7}\partial_{8}+2x_{10}\partial_{8}-x_{9}\partial_{7}-x_{8}\partial_{5}-x_{6}\partial_{4})\otimes id\end{array}$\\\cline{3-3} 
&& $\begin{array}{l}(x_{11}\partial_{13}+x_{9}\partial_{12}+2x_{1}\partial_{9}\partial_{10}+x_{8}\partial_{10}-2x_{7}\partial_{9}-2x_{5}\partial_{8}-x_{4}\partial_{6})\otimes id\end{array}$\\

\newpage

\hline\multirow{8}{*}{$F_4$} & \multirow{8}{*}{15}& $\begin{array}{l}(-x_{5}\partial_{12}^{2}-x_{13}\partial_{12}-x_{3}\partial_{10}^{2}-x_{11}\partial_{10}-x_{7}\partial_{6}-x_{2}\partial_{1})\otimes id\end{array}$\\ \cline{3-3} 
&& $\begin{array}{l}(-x_{5}\partial_{13}^{2}-x_{12}\partial_{13}-x_{3}\partial_{11}^{2}-x_{10}\partial_{11}-x_{6}\partial_{7}-x_{1}\partial_{2})\otimes id\end{array}$\\ \cline{3-3} 
&& $\begin{array}{l}(-x_{7}\partial_{13}^{2}-x_{14}\partial_{13}-x_{1}\partial_{9}^{2}-x_{10}\partial_{9}-x_{6}\partial_{5}-x_{3}\partial_{2})\otimes id\end{array}$\\ \cline{3-3} 
&& $\begin{array}{l}(-x_{7}\partial_{14}^{2}-x_{13}\partial_{14}-x_{1}\partial_{10}^{2}-x_{9}\partial_{10}-x_{5}\partial_{6}-x_{2}\partial_{3})\otimes id\end{array}$\\ \cline{3-3} 
&& $\begin{array}{l}(-x_{15}\partial_{14}-2x_{5}\partial_{11}\partial_{12}+x_{4}\partial_{10}\partial_{11}-x_{13}\partial_{11}-x_{12}\partial_{10}-x_{9}\partial_{8}-2x_{5}\partial_{4}-x_{4}\partial_{3})\otimes id\end{array}$\\ \cline{3-3} 
&& $\begin{array}{l}(-x_{14}\partial_{15}+x_{4}\partial_{12}\partial_{13}-x_{11}\partial_{13}-2x_{3}\partial_{11}\partial_{12}-x_{10}\partial_{12}-x_{8}\partial_{9}-x_{4}\partial_{5}-2x_{3}\partial_{4})\otimes id\end{array}$\\ \cline{3-3} 
&& $\begin{array}{l}(2x_{7}\partial_{11}\partial_{15}^{2}+2x_{6}\partial_{10}\partial_{15}^{2}+2x_{5}\partial_{9}\partial_{15}^{2}-x_{15}\partial_{15}^{2}-2x_{7}\partial_{13}\partial_{14}\partial_{15}-2x_{6}\partial_{12}\partial_{14}\partial_{15}+2x_{4}\partial_{9}\partial_{14}\partial_{15}-x_{14}\partial_{14}\partial_{15} \\ -2x_{5}\partial_{12}\partial_{13}\partial_{15}-2x_{4}\partial_{10}\partial_{13}\partial_{15}-x_{13}\partial_{13}\partial_{15}+2x_{4}\partial_{11}\partial_{12}\partial_{15}-x_{12}\partial_{12}\partial_{15}-4x_{3}\partial_{10}\partial_{11}\partial_{15}-4x_{2}\partial_{9}\partial_{11}\partial_{15}-2x_{11}\partial_{11}\partial_{15} \\ -4x_{1}\partial_{9}\partial_{10}\partial_{15}-2x_{10}\partial_{10}\partial_{15}-2x_{9}\partial_{9}\partial_{15}+2x_{3}\partial_{9}\partial_{14}^{2}+4x_{2}\partial_{9}\partial_{13}\partial_{14}+x_{11}\partial_{13}\partial_{14}+2x_{3}\partial_{11}\partial_{12}\partial_{14}+4x_{1}\partial_{9}\partial_{12}\partial_{14} \\ +x_{10}\partial_{12}\partial_{14}-2x_{8}\partial_{9}\partial_{14}-2x_{2}\partial_{10}\partial_{13}^{2}+2x_{2}\partial_{11}\partial_{12}\partial_{13}-2x_{1}\partial_{10}\partial_{12}\partial_{13}+x_{9}\partial_{12}\partial_{13}+2x_{8}\partial_{10}\partial_{13}+2x_{1}\partial_{11}\partial_{12}^{2}-2x_{8}\partial_{11}\partial_{12} \\ -2x_{3}\partial_{7}\partial_{10}-2x_{2}\partial_{7}\partial_{9}-2x_{1}\partial_{6}\partial_{9}-x_{11}\partial_{7}-x_{10}\partial_{6}-x_{9}\partial_{5}-x_{8}\partial_{4})\otimes id\end{array}$\\ \cline{3-3} 
&& $\begin{array}{l}(-2x_{7}\partial_{11}-2x_{6}\partial_{10}-2x_{5}\partial_{9}-x_{4}\partial_{8}+x_{15})\otimes id\end{array}$\\ \hline

\hline\multirow{12}{*}{$E_6$} & \multirow{12}{*}{16}& $\begin{array}{l}(-x_{16}\partial_{16}^{2}-x_{15}\partial_{15}\partial_{16}-x_{14}\partial_{14}\partial_{16}-x_{13}\partial_{13}\partial_{16}-x_{12}\partial_{12}\partial_{16}-x_{11}\partial_{11}\partial_{16}-x_{10}\partial_{10}\partial_{16}-x_{9}\partial_{9}\partial_{16}-x_{8}\partial_{8}\partial_{16}-x_{6}\partial_{6}\partial_{16} \\ -x_{4}\partial_{4}\partial_{16}+x_{7}\partial_{9}\partial_{15}+x_{5}\partial_{6}\partial_{15}+x_{3}\partial_{4}\partial_{15}+x_{7}\partial_{11}\partial_{14}+x_{5}\partial_{8}\partial_{14}-x_{2}\partial_{4}\partial_{14}-x_{7}\partial_{12}\partial_{13}-x_{5}\partial_{10}\partial_{13}-x_{1}\partial_{4}\partial_{13} \\ -x_{3}\partial_{8}\partial_{12}-x_{2}\partial_{6}\partial_{12}+x_{3}\partial_{10}\partial_{11}-x_{1}\partial_{6}\partial_{11}+x_{2}\partial_{9}\partial_{10}+x_{1}\partial_{8}\partial_{9})\otimes id\end{array}$\\ \cline{3-3} 
&& $\begin{array}{l}x_{16}\otimes id\end{array}$\\ \cline{3-3} 
&& $\begin{array}{l}(-x_{14}\partial_{13}-x_{12}\partial_{11}-x_{10}\partial_{8}-x_{2}\partial_{1})\otimes id\end{array}$\\ \cline{3-3} 
&& $\begin{array}{l}(-x_{13}\partial_{14}-x_{11}\partial_{12}-x_{8}\partial_{10}-x_{1}\partial_{2})\otimes id\end{array}$\\ \cline{3-3} 
&& $\begin{array}{l}(x_{16}\partial_{15}-x_{9}\partial_{7}-x_{6}\partial_{5}-x_{4}\partial_{3})\otimes id\end{array}$\\ \cline{3-3} 
&& $\begin{array}{l}(x_{15}\partial_{16}-x_{7}\partial_{9}-x_{5}\partial_{6}-x_{3}\partial_{4})\otimes id\end{array}$\\ \cline{3-3} 
&& $\begin{array}{l}(x_{15}\partial_{14}-x_{11}\partial_{9}-x_{8}\partial_{6}-x_{3}\partial_{2})\otimes id\end{array}$\\ \cline{3-3} 
&& $\begin{array}{l}(x_{14}\partial_{15}-x_{9}\partial_{11}-x_{6}\partial_{8}-x_{2}\partial_{3})\otimes id\end{array}$\\ \cline{3-3} 
&& $\begin{array}{l}(x_{14}\partial_{12}+x_{13}\partial_{11}-x_{6}\partial_{4}-x_{5}\partial_{3})\otimes id\end{array}$\\ \cline{3-3} 
&& $\begin{array}{l}(x_{12}\partial_{14}+x_{11}\partial_{13}-x_{4}\partial_{6}-x_{3}\partial_{5})\otimes id\end{array}$\\ \cline{3-3} 
&& $\begin{array}{l}(x_{12}\partial_{10}+x_{11}\partial_{8}+x_{9}\partial_{6}+x_{7}\partial_{5})\otimes id\end{array}$\\ \cline{3-3} 
&& $\begin{array}{l}(x_{10}\partial_{12}+x_{8}\partial_{11}+x_{6}\partial_{9}+x_{5}\partial_{7})\otimes id\end{array}$\\ \hline

\newpage
\multirow{14}{*}{$E_7$} & \multirow{14}{*}{27}& $\begin{array}{l}(-x_{22}\partial_{21}-x_{20}\partial_{19}-x_{18}\partial_{17}-x_{16}\partial_{14}-x_{13}\partial_{11}-x_{2}\partial_{1})\otimes id\end{array}$\\ \cline{3-3} 
&& $\begin{array}{l}(-x_{21}\partial_{22}-x_{19}\partial_{20}-x_{17}\partial_{18}-x_{14}\partial_{16}-x_{11}\partial_{13}-x_{1}\partial_{2})\otimes id\end{array}$\\ \cline{3-3} 
&& $\begin{array}{l}(-x_{24}\partial_{23}-x_{22}\partial_{20}-x_{21}\partial_{19}-x_{10}\partial_{8}-x_{7}\partial_{6}-x_{5}\partial_{4})\otimes id\end{array}$\\ \cline{3-3} 
&& $\begin{array}{l}(-x_{23}\partial_{24}-x_{20}\partial_{22}-x_{19}\partial_{21}-x_{8}\partial_{10}-x_{6}\partial_{7}-x_{4}\partial_{5})\otimes id\end{array}$\\ \cline{3-3} 
&& $\begin{array}{l}(-x_{24}\partial_{22}-x_{23}\partial_{20}-x_{17}\partial_{15}-x_{14}\partial_{12}-x_{11}\partial_{9}-x_{3}\partial_{2})\otimes id\end{array}$\\ \cline{3-3} 
&& $\begin{array}{l}(-x_{22}\partial_{24}-x_{20}\partial_{23}-x_{15}\partial_{17}-x_{12}\partial_{14}-x_{9}\partial_{11}-x_{2}\partial_{3})\otimes id\end{array}$\\ \cline{3-3} 
&& $\begin{array}{l}(-x_{25}\partial_{24}-x_{20}\partial_{18}-x_{19}\partial_{17}-x_{12}\partial_{10}-x_{9}\partial_{7}-x_{4}\partial_{3})\otimes id\end{array}$\\ \cline{3-3} 
&& $\begin{array}{l}(-x_{24}\partial_{25}-x_{18}\partial_{20}-x_{17}\partial_{19}-x_{10}\partial_{12}-x_{7}\partial_{9}-x_{3}\partial_{4})\otimes id\end{array}$\\ \cline{3-3} 
&& $\begin{array}{l}(-x_{26}\partial_{25}-x_{18}\partial_{16}-x_{17}\partial_{14}-x_{15}\partial_{12}-x_{7}\partial_{5}-x_{6}\partial_{4})\otimes id\end{array}$\\ \cline{3-3} 
&& $\begin{array}{l}(-x_{25}\partial_{26}-x_{16}\partial_{18}-x_{14}\partial_{17}-x_{12}\partial_{15}-x_{5}\partial_{7}-x_{4}\partial_{6})\otimes id\end{array}$\\ \cline{3-3} 
&& $\begin{array}{l}(-x_{27}\partial_{26}-x_{16}\partial_{13}-x_{14}\partial_{11}-x_{12}\partial_{9}-x_{10}\partial_{7}-x_{8}\partial_{6})\otimes id\end{array}$\\ \cline{3-3} 
&& $\begin{array}{l}(-x_{26}\partial_{27}-x_{13}\partial_{16}-x_{11}\partial_{14}-x_{9}\partial_{12}-x_{7}\partial_{10}-x_{6}\partial_{8})\otimes id\end{array}$\\ \cline{3-3} 
&& $\begin{array}{l}(-x_{27}\partial_{27}^{2}-x_{26}\partial_{26}\partial_{27}-x_{25}\partial_{25}\partial_{27}-x_{24}\partial_{24}\partial_{27}-x_{23}\partial_{23}\partial_{27}-x_{22}\partial_{22}\partial_{27}-x_{21}\partial_{21}\partial_{27}-x_{20}\partial_{20}\partial_{27}-x_{19}\partial_{19}\partial_{27}-x_{18}\partial_{18}\partial_{27} \\ -x_{17}\partial_{17}\partial_{27}-x_{16}\partial_{16}\partial_{27}-x_{15}\partial_{15}\partial_{27}-x_{14}\partial_{14}\partial_{27}-x_{12}\partial_{12}\partial_{27}-x_{10}\partial_{10}\partial_{27}-x_{8}\partial_{8}\partial_{27}-x_{13}\partial_{16}\partial_{26}-x_{11}\partial_{14}\partial_{26}-x_{9}\partial_{12}\partial_{26} \\ -x_{7}\partial_{10}\partial_{26}-x_{6}\partial_{8}\partial_{26}+x_{13}\partial_{18}\partial_{25}+x_{11}\partial_{17}\partial_{25}+x_{9}\partial_{15}\partial_{25}-x_{5}\partial_{10}\partial_{25}-x_{4}\partial_{8}\partial_{25}-x_{13}\partial_{20}\partial_{24}-x_{11}\partial_{19}\partial_{24}+x_{7}\partial_{15}\partial_{24} \\ +x_{5}\partial_{12}\partial_{24}-x_{3}\partial_{8}\partial_{24}+x_{13}\partial_{22}\partial_{23}+x_{11}\partial_{21}\partial_{23}+x_{6}\partial_{15}\partial_{23}+x_{4}\partial_{12}\partial_{23}+x_{3}\partial_{10}\partial_{23}-x_{9}\partial_{19}\partial_{22}-x_{7}\partial_{17}\partial_{22}-x_{5}\partial_{14}\partial_{22} \\ -x_{2}\partial_{8}\partial_{22}+x_{9}\partial_{20}\partial_{21}+x_{7}\partial_{18}\partial_{21}+x_{5}\partial_{16}\partial_{21}-x_{1}\partial_{8}\partial_{21}-x_{6}\partial_{17}\partial_{20}-x_{4}\partial_{14}\partial_{20}+x_{2}\partial_{10}\partial_{20}+x_{6}\partial_{18}\partial_{19}+x_{4}\partial_{16}\partial_{19} \\ +x_{1}\partial_{10}\partial_{19}-x_{3}\partial_{14}\partial_{18}-x_{2}\partial_{12}\partial_{18}+x_{3}\partial_{16}\partial_{17}-x_{1}\partial_{12}\partial_{17}+x_{2}\partial_{15}\partial_{16}+x_{1}\partial_{14}\partial_{15})\otimes id\end{array}$\\ \cline{3-3} 
&& $\begin{array}{l}x_{27}\otimes id\end{array}$\\ \hline 

\newpage
\hline\multirow{16}{*}{$E_8$} & \multirow{16}{*}{57}& $\begin{array}{l}(-x_{51}\partial_{50}-x_{49}\partial_{48}-x_{47}\partial_{46}-x_{45}\partial_{43}-x_{42}\partial_{40}-x_{39}\partial_{36}-x_{23}\partial_{20}-x_{19}\partial_{17}-x_{16}\partial_{14}-x_{13}\partial_{12}-x_{11}\partial_{10}-x_{9}\partial_{8})\otimes id\end{array}$\\ \cline{3-3} 
&& $\begin{array}{l}(-x_{50}\partial_{51}-x_{48}\partial_{49}-x_{46}\partial_{47}-x_{43}\partial_{45}-x_{40}\partial_{42}-x_{36}\partial_{39}-x_{20}\partial_{23}-x_{17}\partial_{19}-x_{14}\partial_{16}-x_{12}\partial_{13}-x_{10}\partial_{11}-x_{8}\partial_{9})\otimes id\end{array}$\\ \cline{3-3} 
&& $\begin{array}{l}(-x_{53}\partial_{52}-x_{51}\partial_{49}-x_{50}\partial_{48}-x_{38}\partial_{35}-x_{34}\partial_{32}-x_{1}\partial_{28}\partial_{29}-x_{31}\partial_{29}-x_{30}\partial_{28}-x_{27}\partial_{25}-x_{24}\partial_{21}-x_{11}\partial_{9}-x_{10}\partial_{8}-x_{7}\partial_{6})\otimes id\end{array}$\\ \cline{3-3} 
&& $\begin{array}{l}(-x_{52}\partial_{53}-x_{49}\partial_{51}-x_{48}\partial_{50}-x_{35}\partial_{38}-x_{32}\partial_{34}-x_{1}\partial_{30}\partial_{31}-x_{29}\partial_{31}-x_{28}\partial_{30}-x_{25}\partial_{27}-x_{21}\partial_{24}-x_{9}\partial_{11}-x_{8}\partial_{10}-x_{6}\partial_{7})\otimes id\end{array}$\\ \cline{3-3} 
&& $\begin{array}{l}(-x_{53}\partial_{51}-x_{52}\partial_{49}-x_{46}\partial_{44}-x_{43}\partial_{41}-x_{40}\partial_{37}-x_{36}\partial_{33}-x_{26}\partial_{23}-x_{22}\partial_{19}-x_{18}\partial_{16}-x_{15}\partial_{13}-x_{10}\partial_{7}-x_{8}\partial_{6})\otimes id\end{array}$\\ \cline{3-3} 
&& $\begin{array}{l}(-x_{51}\partial_{53}-x_{49}\partial_{52}-x_{44}\partial_{46}-x_{41}\partial_{43}-x_{37}\partial_{40}-x_{33}\partial_{36}-x_{23}\partial_{26}-x_{19}\partial_{22}-x_{16}\partial_{18}-x_{13}\partial_{15}-x_{7}\partial_{10}-x_{6}\partial_{8})\otimes id\end{array}$\\ \cline{3-3} 
&& $\begin{array}{l}(-x_{54}\partial_{53}-x_{49}\partial_{47}-x_{48}\partial_{46}-x_{41}\partial_{38}-x_{37}\partial_{34}-x_{33}\partial_{30}-x_{29}\partial_{26}-x_{25}\partial_{22}-x_{21}\partial_{18}-x_{13}\partial_{11}-x_{12}\partial_{10}-x_{6}\partial_{5})\otimes id\end{array}$\\ \cline{3-3} 
&& $\begin{array}{l}(-x_{53}\partial_{54}-x_{47}\partial_{49}-x_{46}\partial_{48}-x_{38}\partial_{41}-x_{34}\partial_{37}-x_{30}\partial_{33}-x_{26}\partial_{29}-x_{22}\partial_{25}-x_{18}\partial_{21}-x_{11}\partial_{13}-x_{10}\partial_{12}-x_{5}\partial_{6})\otimes id\end{array}$\\ \cline{3-3} 
&& $\begin{array}{l}(-x_{55}\partial_{54}-x_{47}\partial_{45}-x_{46}\partial_{43}-x_{44}\partial_{41}-x_{34}\partial_{31}-x_{1}\partial_{27}\partial_{29}-x_{32}\partial_{29}-x_{30}\partial_{27}-x_{28}\partial_{25}-x_{18}\partial_{15}-x_{16}\partial_{13}-x_{14}\partial_{12}-x_{5}\partial_{4})\otimes id\end{array}$\\ \cline{3-3} 
&& $\begin{array}{l}(-x_{54}\partial_{55}-x_{45}\partial_{47}-x_{43}\partial_{46}-x_{41}\partial_{44}-x_{31}\partial_{34}-x_{1}\partial_{30}\partial_{32}-x_{29}\partial_{32}-x_{27}\partial_{30}-x_{25}\partial_{28}-x_{15}\partial_{18}-x_{13}\partial_{16}-x_{12}\partial_{14}-x_{4}\partial_{5})\otimes id\end{array}$\\ \cline{3-3} 
&& $\begin{array}{l}(-x_{56}\partial_{55}-x_{45}\partial_{42}-x_{43}\partial_{40}-x_{41}\partial_{37}-x_{38}\partial_{34}-x_{35}\partial_{32}-x_{27}\partial_{24}-x_{25}\partial_{21}-x_{22}\partial_{18}-x_{19}\partial_{16}-x_{17}\partial_{14}-x_{4}\partial_{3})\otimes id\end{array}$\\ \cline{3-3} 
&& $\begin{array}{l}(-x_{55}\partial_{56}-x_{42}\partial_{45}-x_{40}\partial_{43}-x_{37}\partial_{41}-x_{34}\partial_{38}-x_{32}\partial_{35}-x_{24}\partial_{27}-x_{21}\partial_{25}-x_{18}\partial_{22}-x_{16}\partial_{19}-x_{14}\partial_{17}-x_{3}\partial_{4})\otimes id\end{array}$\\ \cline{3-3} 
&& $\begin{array}{l}(-x_{57}\partial_{56}-x_{42}\partial_{39}-x_{40}\partial_{36}-x_{37}\partial_{33}-x_{34}\partial_{30}-x_{1}\partial_{27}\partial_{28}-x_{32}\partial_{28}-x_{31}\partial_{27}-x_{29}\partial_{25}-x_{26}\partial_{22}-x_{23}\partial_{19}-x_{20}\partial_{17}-x_{3}\partial_{2})\otimes id\end{array}$\\ \cline{3-3} 
&& $\begin{array}{l}(-x_{56}\partial_{57}-x_{39}\partial_{42}-x_{36}\partial_{40}-x_{33}\partial_{37}-x_{30}\partial_{34}-x_{1}\partial_{31}\partial_{32}-x_{28}\partial_{32}-x_{27}\partial_{31}-x_{25}\partial_{29}-x_{22}\partial_{26}-x_{19}\partial_{23}-x_{17}\partial_{20}-x_{2}\partial_{3})\otimes id\end{array}$\\ \cline{3-3} 
&& $\begin{array}{l}(-x_{57}\partial_{57}^{2}-x_{56}\partial_{56}\partial_{57}-x_{55}\partial_{55}\partial_{57}-x_{54}\partial_{54}\partial_{57}-x_{53}\partial_{53}\partial_{57}-x_{52}\partial_{52}\partial_{57}-x_{51}\partial_{51}\partial_{57}-x_{50}\partial_{50}\partial_{57}-x_{49}\partial_{49}\partial_{57}-x_{48}\partial_{48}\partial_{57} \\ -x_{47}\partial_{47}\partial_{57}-x_{46}\partial_{46}\partial_{57}-x_{45}\partial_{45}\partial_{57}-x_{44}\partial_{44}\partial_{57}-x_{43}\partial_{43}\partial_{57}-x_{42}\partial_{42}\partial_{57}-x_{41}\partial_{41}\partial_{57}-x_{40}\partial_{40}\partial_{57}+x_{1}\partial_{20}\partial_{39}\partial_{57} \\ -x_{38}\partial_{38}\partial_{57}-x_{37}\partial_{37}\partial_{57}-x_{1}\partial_{23}\partial_{36}\partial_{57}-x_{35}\partial_{35}\partial_{57}-x_{34}\partial_{34}\partial_{57}+x_{1}\partial_{26}\partial_{33}\partial_{57}-x_{32}\partial_{32}\partial_{57}-x_{31}\partial_{31}\partial_{57}-x_{1}\partial_{29}\partial_{30}\partial_{57} \\ -x_{29}\partial_{29}\partial_{57}-x_{26}\partial_{26}\partial_{57}-x_{23}\partial_{23}\partial_{57}-x_{20}\partial_{20}\partial_{57}-x_{1}\partial_{20}\partial_{42}\partial_{56}-x_{39}\partial_{42}\partial_{56}+x_{1}\partial_{23}\partial_{40}\partial_{56}-x_{36}\partial_{40}\partial_{56}-x_{1}\partial_{26}\partial_{37}\partial_{56} \\ -x_{33}\partial_{37}\partial_{56}+x_{1}\partial_{29}\partial_{34}\partial_{56}-x_{30}\partial_{34}\partial_{56}-2x_{1}\partial_{31}\partial_{32}\partial_{56}-x_{28}\partial_{32}\partial_{56}-x_{27}\partial_{31}\partial_{56}-x_{25}\partial_{29}\partial_{56}-x_{22}\partial_{26}\partial_{56}-x_{19}\partial_{23}\partial_{56} \\ -x_{17}\partial_{20}\partial_{56}+x_{1}\partial_{20}\partial_{45}\partial_{55}+x_{39}\partial_{45}\partial_{55}-x_{1}\partial_{23}\partial_{43}\partial_{55}+x_{36}\partial_{43}\partial_{55}+x_{1}\partial_{26}\partial_{41}\partial_{55}+x_{33}\partial_{41}\partial_{55}-x_{1}\partial_{29}\partial_{38}\partial_{55} \\ +x_{30}\partial_{38}\partial_{55}+2x_{1}\partial_{31}\partial_{35}\partial_{55}+x_{28}\partial_{35}\partial_{55}-x_{24}\partial_{31}\partial_{55}-x_{21}\partial_{29}\partial_{55}-x_{18}\partial_{26}\partial_{55}-x_{16}\partial_{23}\partial_{55}-x_{14}\partial_{20}\partial_{55}-x_{1}\partial_{20}\partial_{47}\partial_{54} \\ -x_{39}\partial_{47}\partial_{54}+x_{1}\partial_{23}\partial_{46}\partial_{54}-x_{36}\partial_{46}\partial_{54}-x_{1}\partial_{26}\partial_{44}\partial_{54}-x_{33}\partial_{44}\partial_{54}+2x_{1}\partial_{32}\partial_{38}\partial_{54}+x_{27}\partial_{38}\partial_{54}-2x_{1}\partial_{34}\partial_{35}\partial_{54} \\ +x_{25}\partial_{35}\partial_{54}+x_{24}\partial_{34}\partial_{54}+x_{21}\partial_{32}\partial_{54}-x_{15}\partial_{26}\partial_{54}-x_{13}\partial_{23}\partial_{54}-x_{12}\partial_{20}\partial_{54}+x_{1}\partial_{20}\partial_{49}\partial_{53}+x_{39}\partial_{49}\partial_{53}-x_{1}\partial_{23}\partial_{48}\partial_{53} \\ +x_{36}\partial_{48}\partial_{53}+x_{1}\partial_{29}\partial_{44}\partial_{53}-x_{30}\partial_{44}\partial_{53}-2x_{1}\partial_{32}\partial_{41}\partial_{53}-x_{27}\partial_{41}\partial_{53}+2x_{1}\partial_{35}\partial_{37}\partial_{53}-x_{24}\partial_{37}\partial_{53}+x_{22}\partial_{35}\partial_{53} \\ +x_{18}\partial_{32}\partial_{53}+x_{15}\partial_{29}\partial_{53}-x_{11}\partial_{23}\partial_{53}-x_{10}\partial_{20}\partial_{53}-x_{1}\partial_{20}\partial_{51}\partial_{52}-x_{39}\partial_{51}\partial_{52}+x_{1}\partial_{23}\partial_{50}\partial_{52}-x_{36}\partial_{50}\partial_{52}-2x_{1}\partial_{31}\partial_{44}\partial_{52} \\ -x_{28}\partial_{44}\partial_{52}+2x_{1}\partial_{34}\partial_{41}\partial_{52}-x_{25}\partial_{41}\partial_{52}-2x_{1}\partial_{37}\partial_{38}\partial_{52}-x_{22}\partial_{38}\partial_{52}-x_{21}\partial_{37}\partial_{52}-x_{18}\partial_{34}\partial_{52}-x_{15}\partial_{31}\partial_{52}-x_{9}\partial_{23}\partial_{52} \\ -x_{8}\partial_{20}\partial_{52}+x_{1}\partial_{26}\partial_{48}\partial_{51}+x_{33}\partial_{48}\partial_{51}-x_{1}\partial_{29}\partial_{46}\partial_{51}+x_{30}\partial_{46}\partial_{51}+2x_{1}\partial_{32}\partial_{43}\partial_{51}+x_{27}\partial_{43}\partial_{51}-2x_{1}\partial_{35}\partial_{40}\partial_{51} \\ +x_{24}\partial_{40}\partial_{51}+x_{19}\partial_{35}\partial_{51}+x_{16}\partial_{32}\partial_{51}+x_{13}\partial_{29}\partial_{51}+x_{11}\partial_{26}\partial_{51}-x_{7}\partial_{20}\partial_{51}-x_{1}\partial_{26}\partial_{49}\partial_{50}-x_{33}\partial_{49}\partial_{50}+x_{1}\partial_{29}\partial_{47}\partial_{50} \\ -x_{30}\partial_{47}\partial_{50}-2x_{1}\partial_{32}\partial_{45}\partial_{50}-x_{27}\partial_{45}\partial_{50}+2x_{1}\partial_{35}\partial_{42}\partial_{50}-x_{24}\partial_{42}\partial_{50}+x_{17}\partial_{35}\partial_{50}+x_{14}\partial_{32}\partial_{50}+x_{12}\partial_{29}\partial_{50}+x_{10}\partial_{26}\partial_{50} \\ +x_{7}\partial_{23}\partial_{50}+2x_{1}\partial_{31}\partial_{46}\partial_{49}+x_{28}\partial_{46}\partial_{49}-2x_{1}\partial_{34}\partial_{43}\partial_{49}+x_{25}\partial_{43}\partial_{49}+2x_{1}\partial_{38}\partial_{40}\partial_{49}+x_{21}\partial_{40}\partial_{49}-x_{19}\partial_{38}\partial_{49} \\ -x_{16}\partial_{34}\partial_{49}-x_{13}\partial_{31}\partial_{49}+x_{9}\partial_{26}\partial_{49}-x_{6}\partial_{20}\partial_{49}-2x_{1}\partial_{31}\partial_{47}\partial_{48}-x_{28}\partial_{47}\partial_{48}+2x_{1}\partial_{34}\partial_{45}\partial_{48}-x_{25}\partial_{45}\partial_{48}-2x_{1}\partial_{38}\partial_{42}\partial_{48} \\ -x_{21}\partial_{42}\partial_{48}-x_{17}\partial_{38}\partial_{48}-x_{14}\partial_{34}\partial_{48}-x_{12}\partial_{31}\partial_{48}+x_{8}\partial_{26}\partial_{48}+x_{6}\partial_{23}\partial_{48}+2x_{1}\partial_{37}\partial_{43}\partial_{47}+x_{22}\partial_{43}\partial_{47}-2x_{1}\partial_{40}\partial_{41}\partial_{47} \\ +x_{19}\partial_{41}\partial_{47}+x_{18}\partial_{40}\partial_{47}+x_{16}\partial_{37}\partial_{47}-x_{11}\partial_{31}\partial_{47}-x_{9}\partial_{29}\partial_{47}-x_{5}\partial_{20}\partial_{47}-2x_{1}\partial_{37}\partial_{45}\partial_{46}-x_{22}\partial_{45}\partial_{46}+2x_{1}\partial_{41}\partial_{42}\partial_{46} \\ -x_{18}\partial_{42}\partial_{46}+x_{17}\partial_{41}\partial_{46}+x_{14}\partial_{37}\partial_{46}-x_{10}\partial_{31}\partial_{46}-x_{8}\partial_{29}\partial_{46}+x_{5}\partial_{23}\partial_{46}+2x_{1}\partial_{40}\partial_{44}\partial_{45}-x_{19}\partial_{44}\partial_{45}+x_{15}\partial_{40}\partial_{45} \\ +x_{13}\partial_{37}\partial_{45}+x_{11}\partial_{34}\partial_{45}+x_{9}\partial_{32}\partial_{45}-x_{4}\partial_{20}\partial_{45}-2x_{1}\partial_{42}\partial_{43}\partial_{44}-x_{17}\partial_{43}\partial_{44}-x_{16}\partial_{42}\partial_{44}-x_{14}\partial_{40}\partial_{44}-x_{7}\partial_{31}\partial_{44} \\ -x_{6}\partial_{29}\partial_{44}-x_{5}\partial_{26}\partial_{44}-x_{15}\partial_{42}\partial_{43}+x_{12}\partial_{37}\partial_{43}+x_{10}\partial_{34}\partial_{43}+x_{8}\partial_{32}\partial_{43}+x_{4}\partial_{23}\partial_{43}-x_{13}\partial_{41}\partial_{42}-x_{11}\partial_{38}\partial_{42}-x_{9}\partial_{35}\partial_{42} \\ -x_{3}\partial_{20}\partial_{42}-x_{12}\partial_{40}\partial_{41}+x_{7}\partial_{34}\partial_{41}+x_{6}\partial_{32}\partial_{41}-x_{4}\partial_{26}\partial_{41}-x_{10}\partial_{38}\partial_{40}-x_{8}\partial_{35}\partial_{40}+x_{3}\partial_{23}\partial_{40}-x_{2}\partial_{20}\partial_{39}-x_{7}\partial_{37}\partial_{38} \\ +x_{5}\partial_{32}\partial_{38}+x_{4}\partial_{29}\partial_{38}-x_{6}\partial_{35}\partial_{37}-x_{3}\partial_{26}\partial_{37}+x_{2}\partial_{23}\partial_{36}-x_{5}\partial_{34}\partial_{35}-x_{4}\partial_{31}\partial_{35}+x_{3}\partial_{29}\partial_{34}-x_{2}\partial_{26}\partial_{33}-x_{3}\partial_{31}\partial_{32} \\ +x_{2}\partial_{29}\partial_{30}-x_{2}\partial_{1})\otimes id\end{array}$\\ \cline{3-3} 
&& $\begin{array}{l}(-x_{1}\partial_{2}+x_{57})\otimes id\end{array}$\\ \hline 
\end{longtable}

\begin{table}\caption{\label{tableGtwoStructure} The Lie bracket pairing table of $G_2$. Generators are indexed according to root spaces, and root spaces are indexed according to graded lexicographic order. The first simple root is short. }
\begin{tiny}
\[\begin{array}{cc|ccccccccccccccc}
\mathrm{roots}&\varepsilon-\mathrm{~notation}&[\bullet, \bullet]
& g_{-6} & g_{-5} & g_{-4} & g_{-3} & g_{-2} & g_{-1} & h_{1} & h_{2} & g_{1} & g_{2} & g_{3} & g_{4} & g_{5} & g_{6}\\
(-3, -2)&\varepsilon_{1}+\varepsilon_{2}-2\varepsilon_{3}&g_{-6}& 0& 0& 0& 0& 0& 0& 0& 3g_{-6}& 0& g_{-5}& -g_{-4}& g_{-3}& -g_{-2}& -h_{1}-2/3h_{2}\\
(-3, -1)&-\varepsilon_{1}+2\varepsilon_{2}-\varepsilon_{3}&g_{-5}& 0& 0& 0& 0& g_{-6}& 0& 3g_{-5}& -3g_{-5}& g_{-4}& 0& 0& -g_{-1}& -h_{1}-1/3h_{2}& -g_{2}\\
(-2, -1)&\varepsilon_{2}-\varepsilon_{3}&g_{-4}& 0& 0& 0& -3g_{-6}& 0& 3g_{-5}& g_{-4}& 0& 2g_{-3}& 0& -2g_{-1}& -2h_{1}-h_{2}& -g_{1}& g_{3}\\
(-1, -1)&\varepsilon_{1}-\varepsilon_{3}&g_{-3}& 0& 0& 3g_{-6}& 0& 0& 2g_{-4}& -g_{-3}& 3g_{-3}& 3g_{-2}& -g_{-1}& -h_{1}-h_{2}& -2g_{1}& 0& -g_{4}\\
(0, -1)&2\varepsilon_{1}-\varepsilon_{2}-\varepsilon_{3}&g_{-2}& 0& -g_{-6}& 0& 0& 0& g_{-3}& -3g_{-2}& 6g_{-2}& 0& -1/3h_{2}& -g_{1}& 0& 0& g_{5}\\
(-1, 0)&-\varepsilon_{1}+\varepsilon_{2}&g_{-1}& 0& 0& -3g_{-5}& -2g_{-4}& -g_{-3}& 0& 2g_{-1}& -3g_{-1}& -h_{1}& 0& 3g_{2}& 2g_{3}& g_{4}& 0\\
(0, 0)&0&h_{1}& 0& -3g_{-5}& -g_{-4}& g_{-3}& 3g_{-2}& -2g_{-1}& 0& 0& 2g_{1}& -3g_{2}& -g_{3}& g_{4}& 3g_{5}& 0\\
(0, 0)&0&h_{2}& -3g_{-6}& 3g_{-5}& 0& -3g_{-3}& -6g_{-2}& 3g_{-1}& 0& 0& -3g_{1}& 6g_{2}& 3g_{3}& 0& -3g_{5}& 3g_{6}\\
(1, 0)&\varepsilon_{1}-\varepsilon_{2}&g_{1}& 0& -g_{-4}& -2g_{-3}& -3g_{-2}& 0& h_{1}& -2g_{1}& 3g_{1}& 0& g_{3}& 2g_{4}& 3g_{5}& 0& 0\\
(0, 1)&-2\varepsilon_{1}+\varepsilon_{2}+\varepsilon_{3}&g_{2}& -g_{-5}& 0& 0& g_{-1}& 1/3h_{2}& 0& 3g_{2}& -6g_{2}& -g_{3}& 0& 0& 0& g_{6}& 0\\
(1, 1)&-\varepsilon_{1}+\varepsilon_{3}&g_{3}& g_{-4}& 0& 2g_{-1}& h_{1}+h_{2}& g_{1}& -3g_{2}& g_{3}& -3g_{3}& -2g_{4}& 0& 0& -3g_{6}& 0& 0\\
(2, 1)&-\varepsilon_{2}+\varepsilon_{3}&g_{4}& -g_{-3}& g_{-1}& 2h_{1}+h_{2}& 2g_{1}& 0& -2g_{3}& -g_{4}& 0& -3g_{5}& 0& 3g_{6}& 0& 0& 0\\
(3, 1)&\varepsilon_{1}-2\varepsilon_{2}+\varepsilon_{3}&g_{5}& g_{-2}& h_{1}+1/3h_{2}& g_{1}& 0& 0& -g_{4}& -3g_{5}& 3g_{5}& 0& -g_{6}& 0& 0& 0& 0\\
(3, 2)&-\varepsilon_{1}-\varepsilon_{2}+2\varepsilon_{3}&g_{6}& h_{1}+2/3h_{2}& g_{2}& -g_{3}& g_{4}& -g_{5}& 0& 0& -3g_{6}& 0& 0& 0& 0& 0& 0\\
\end{array}
\]
\end{tiny}
\end{table}
\end{landscape}
\nocite{Dixmier}
\bibliographystyle{plain}

\begin{thebibliography}{10}

\bibitem{Berline:algebrasDiffOps}
N.~Conze-Berline.
\newblock Algèbres d'opérateurs différentiels et quotients des algèbres
  enveloppantes.
\newblock {\em Bulletin de la Société Mathématique de France}, 102:379--415,
  1974.

\bibitem{DeGraaf:LieAlgebrasAndAlgorithms}
W.~de~Graaf.
\newblock {\em Lie algebras: theory and algorithms}.
\newblock Elsevier, 2000.

\bibitem{Dixmier}
J.~Dixmier.
\newblock {\em Algebres Enveloppantes}.
\newblock Gauthier-Villars Editeur, Paris-Bruxelles-Montreal, 1974.

\bibitem{Humphreys}
J.~Humphreys.
\newblock {\em Introduction to {L}ie algebras and representation theory}.
\newblock Springer-Verlag, New York, 1972.

\bibitem{JacksonMilev:ShapovalovForm}
S.~Jackson and T.~Milev.
\newblock Computing irreducible finite dimensional complex {L}ie algebra
  representations, in preparation.

\bibitem{Joseph:MinimalRealizations}
A.~Joseph.
\newblock Minimal realizations and spectrum generating algebras.
\newblock {\em Comm. Math. Phys.}, 36:325--338, 1974.

\bibitem{Joseph:MinDimNilpotentOrbit}
A.~Joseph.
\newblock A minimal orbit in a simple {L}ie algebra and its associated maximal
  ideal.
\newblock {\em Ann. Sci. Ecole Norm. Sup.}, 9:1--29, 1976.

\bibitem{Littelmann:PathsAndRootOperators}
P.~Littelmann.
\newblock Paths and root operators in representation theory.
\newblock {\em Annals of Mathematics}, 142:499--525, 1995.

\bibitem{Littelmann:ConesCrystalsPatterns}
P.~Littelmann.
\newblock Cones, crystals, and patterns.
\newblock {\em Transf. Groups}, 3:145--179, 1998.

\bibitem{Milev:vpf}
T.~Milev.
\newblock Vector partition function project,
  \url{http://sourceforge.net/projects/vectorpartition/}, {J}anuary 2014.
  {S}ource code, revision 1951,
  \url{http://sourceforge.net/p/vectorpartition/code/1951/tree/trunk/}.
\newblock 2009-2014.

\bibitem{MilevSomberg:so7g2Fmethod}
T.~Milev and P.~Somberg.
\newblock The {F}-method and a branching problem for generalized {V}erma
  modules associated to $({G}_2,{so(7)})$.
\newblock {\em Archivum Mathematicum}, 49:317–332, 2013.

\bibitem{MilevSomberg:BranchingShort}
T.~Milev and P.~Somberg.
\newblock The branching problem for generalized {V}erma modules, with
  application to the pair {($so(7),Lie G_2$)}, {DOI:
  10.1142/S0219498814500340}.
\newblock {\em J. Algebra Appl.}, 13, 2014.

\bibitem{OerstedKobayashiSombergSoucek:Fmethod}
B.~{\O}rsted, T.~Kobayashi, P.~Somberg, and V.~Soucek.
\newblock Branching laws for {V}erma modules and applications in parabolic
  geometry. {I}. \url{http://arxiv.org/abs/1305.6040}.

\bibitem{Samelson}
H.~Samelson.
\newblock {\em Notes on {L}ie algebras}.
\newblock Springer, 2nd edition, 1990.

\end{thebibliography}

\end{document}